\documentclass[13pt]{amsart}
\usepackage{amsmath}
\usepackage{}
\usepackage{amsfonts}
\usepackage{amsfonts,latexsym,rawfonts,amsmath,amssymb,amsthm}
\usepackage[plainpages=false]{hyperref}

\usepackage{pdfsync}

\usepackage{graphicx}

\numberwithin{equation}{section}

\newcommand{\beq}{\begin{equation}}
\newcommand{\eeq}{\end{equation}}
\newcommand{\beqs}{\begin{eqnarray*}}
\newcommand{\eeqs}{\end{eqnarray*}}
\newcommand{\beqn}{\begin{eqnarray}}
\newcommand{\eeqn}{\end{eqnarray}}
\newcommand{\beqa}{\begin{array}}
\newcommand{\eeqa}{\end{array}}

\def\lra{\longrightarrow}

\def\bc{\begin{center}}
\def\ec{\end{center}}

\def\begeq{\begin{equation}}
\def\endeq{\end{equation}}
\def\and{\quad{\rm and}\quad}

\let\lra=\longrightarrow

\def\mapright\#1{\,\smash{\mathop{\lra}\limits^{\#1}}\,}

\newtheorem{prop}{Proposition}[section]
\newtheorem{theo}[prop]{Theorem}
\newtheorem{lem}[prop]{Lemma}

\newtheorem{cor}[prop]{Corollary}
\newtheorem{rem}[prop]{Remark}
\newtheorem{exa}[prop]{Example}
\newtheorem{defi}[prop]{Definition}
\newtheorem{conj}[prop]{Conjecture}

\begin{document}

\title{$G$-Sasaki manifolds and K-energy }


\author{Yan Li and Xiaohua $\text{Zhu}^{*}$}


\address{School of Mathematical Sciences, Peking
University, Beijing 100871, China.}

\email{liyanmath@pku.edu.cn\ \ \ xhzhu@math.pku.edu.cn}

\thanks {* Partially supported by NSFC Grants 11331001 and 11771019.}

\begin{abstract}In this paper, we introduce a class of Sasaki manifolds with a reductive $G$-group action, called
$G$-Sasaki manifolds. By reducing K-energy to a functional defined on a class of convex functions on a
moment polytope, we give a criterion for the properness of K-energy. In particular, we deduce a sufficient and necessary condition related to the polytope for the existence of  $G$-Sasaki Einstein metrics. A similar result is also obtained for  $G$-Sasaki Ricci solitons.
As an application, we construct several  examples of  $G$-Sasaki Ricci solitons by an established openness theorem for  $G$-Sasaki Ricci solitons.

\end{abstract}

\maketitle

\tableofcontents

\setcounter{section}{-1}

\section{Introduction}

In this paper, we introduce a class of $G$-Sasaki manifolds $M$ with a reductive $G$-group action, called
$G$-Sasaki manifolds. The group acts on the K\"ahler cone $C(M)$ of $M$ as a $G\times G$ action, see Definition \ref{definition} for details. One of our motivations is from the fundamental work of Alexeev and Brion in group  compactifications theory  \cite{AB1, AB2}. In general, there are many different compactifications  $\hat G$ of   $G$ with an extended $G\times G$ action, and the compactification space may not be a smooth manifold, perhaps just an algebraic variety.

More recently, the K\"ahler geometry on $G$-manifolds (called for simplicity, if $\hat G$ is  smooth and K\"ahlerian) has been extensively studied (cf. \cite{AK, Del2, Del3,  LZZ, LZ, De17}). For examples, Delcroix proved the existence of K\"ahler-Einstein metrics on a Fano $G$-manifold under a sufficient and necessary condition \cite{Del2}, and later, Li, Zhou and Zhu gave another proof of Delcroix's result  and generalize it to K\"ahler-Ricci solitons \cite{LZZ}. Moreover, Delcroix's condition can be explained in terms of $K$-stability \cite {LZZ} (also see \cite{Del3}),  and thus their results can be both regarded as
direct proofs to Yau-Tian-Danaldson conjecture in case of $G$-manifolds \cite{Yau93,  T1, T4, CDS}.

In Sasaki geometry, a transverse K\"ahler metric is very closely related  to a K\"ahler metric on a complex manifold (cf. \cite{Boyer, Futaki-Ono-Wang}). In particular, if a Sasaki manifold $M$ is regular or quasi-regular, then $M$ is  just an $S^1$-bundle over a K\"ahler manifold or an orbifold. Another relationship is that a transverse  Sasaki-Einstein metric corresponds to  a K\"ahler-Ricci flat cone.  Recently,  Collins and Sz\'ekelyhidi established a link between  transverse  Sasaki-Einstein metrics and  stable K\"ahler cones   as in  the Yau-Tian-Danaldson conjecture  \cite{CS}. The question of existence of Sasaki-Einstein metrics  has received increasing attention
in the physics community through their connection to the AdS/CFT correspondence
(cf. \cite{Ma, Martelli-Sparks-Yau 2, CDSh}).  We refer the reader to see  many interesting examples of such metrics  in a monumental work of Boyer and Galicki \cite{Boyer}.

Our goal in this paper is to extend  the argument  in \cite{LZZ}  on  $G$-manifolds  to $G$-Sasaki manifolds. In particular, we
prove a version of Delcroix's theorem for the existence of transverse  Sasaki-Einstein metrics   in case of  $G$-Sasaki manifolds.  Our  result also generalizes a beautiful theorem of Futaki, Ono and Wang
for the existence of transverse  Sasaki-Ricci solitons on toric Sasaki manifolds \cite{Futaki-Ono-Wang}.

To state our main results, let us introduce some notations for Lie group. Let $G$ be a complex, connected, reductive group of complex dimension $(n+1)$, which is the complexification of a  maximal compact subgroup $K$ of $G$.  Let $T$ be a maximal compact torus of $K$ and $T^c$  its complexification.
 Denote by $\mathfrak g, \mathfrak t$ the Lie algebra of $G$ and $T$, respectively. Set $\mathfrak a=J_G\mathfrak t$,  where $J_G$ is  the complex structure of $G$.   We fix a scalar product $\langle\cdot,\cdot\rangle$ on $\mathfrak a$ which extends the Killing form defined on the semi-simple part $\mathfrak a_{ss}$ of $\mathfrak a$ with persevering $\mathfrak a_{ss}$ orthogonal to the centre $\mathfrak a_z=\mathfrak a\cap\mathfrak z(\mathfrak g)$.
 Denote by $R_G$ the root system of $(G,T^c)$ and choose a system $R_G^+$ of positive roots, which defines a positive Weyl chamber $\mathfrak a_+\subset\mathfrak a$.    Let $\mathfrak a^*$ be  the dual of  $\mathfrak a$  and  $\mathfrak a^*_+$  the dual of $\mathfrak a_+$  under $\langle\cdot,\cdot\rangle$.

Since the K\"ahler cone $C(M)$ of $G$-Sasaki manifold $M$ contains a toric cone $Z$ generated by the torus $T^c$,     there are a  moment  polytope cone $\mathfrak C$ associated to  $Z$  and a  restricted moment polytope $\mathcal P\subset \mathfrak C$  associated to $Z\cap M$, respectively    (cf. Section 2, 3).    Let $\mathcal P_+=\mathcal P\cap\mathfrak a_+^*$. Define a function on $\mathfrak a^*_+$ related to positive roots in $R_G^+$ by
$$\pi(y)=\prod_{\alpha\in R_G^+}\langle y,\alpha\rangle^2.$$
We introduce the barycentre of $\mathcal P_+$ by
\begin{eqnarray*}
bar(\mathcal P_+)=\frac{\int_{\mathcal P_+}y \pi\,d\sigma_c}{\int_{\mathcal P_+}\pi\,d\sigma_c},
\end{eqnarray*}
where $d\sigma_c$ is the Lebesgue measure on $\mathcal P$.

Let $\Xi$  be  the relative interior of the cone generated by $R_G^+$ and set
$$\sigma=\frac{1}{2}\sum_{\alpha\in R_G^+}\alpha.$$
 Then we  state  our first main result as follows.

\begin{theo}\label{main thm}
Let $(M,g)$ be a $(2n+1)$-dimensional $G$-Sasaki manifold with $\omega_g^T\in \frac{\pi}{n+1}c_1^B(M)>0$. Then $M$ admits a transverse Sasaki-Einstein metric if and only if $bar(\mathcal P_+)$
satisfies
\begin{eqnarray}\label{0316}
bar(\mathcal P_+)-{\frac{2}{n+1}}\sigma+{\frac{1}{n+1}}\gamma_0\in\Xi,
\end{eqnarray}
where $\gamma_0$ is a rational vector in $\mathfrak a_z^*$ (the dual of $\mathfrak a_z$) determined in Proposition \ref{Fano condition}.
\end{theo}

(\ref{0316}) is an obstruction to the existence of $G$-Sasaki Einstein metrics. In fact, we will use an argument in \cite{ZZ, ZZ1} to derive an analytic obstruction to the existence of $G$-Sasaki metrics with constant transverse scalar curvature in terms of convex $W$-invariant piecewise linear functions (cf. Proposition \ref{ness}). Then by a construction of  piecewise linear function in \cite{LZZ}, the analytic obstruction implies (\ref{0316}).

For the sufficient part of Theorem \ref{main thm}, we use the argument in \cite{LZZ} to prove the properness of K-energy on the space of $K\times K$-invariant potentials on ${\frac{\pi}{n+1}} c_1^B(M)$ through the reduced K-energy $\mu(\cdot)$. Since $\mu(\cdot)$ is defined on a class of convex functions on $\iota^*(\mathcal P_+)$ and
$ \mathcal P$ does not satisfy the Delzant condition in general \cite{De}, we shall modify the proof of main theorem in \cite[Theorem 1.2]{LZZ}. Here $\iota^*$ is an isomorphism from $\mathcal P_+$ to a polytope $P$ in a subspace of codimension one in $\mathfrak a^*$ (cf. Section 3).
In fact, our method works for any $G$-Sasaki manifold to give a criterion for the properness of K-energy (cf. Theorem \ref{proper general class}). It is interesting to mention that the form of $\mu(\cdot)$ may depend on the choice of $\iota^*$ if its transverse K\"ahler class of $G$-Sasaki manifold is not belonged to a multiple of $c_1^B(M)$ (cf. Remark \ref{mu-depends}).

From the proof in Theorem \ref{main thm}, we actually prove the following strong properness of K-energy $\mathcal K(\cdot)$ for a   $G$-Sasaki Einstein manifold.

\begin{cor}\label{corollary-strong-sasaki}Let $(M, g)$ be a $(2n+1)$-dimensional $G$-Sasaki manifold with $\omega_g^T\in{\frac{\pi}{n+1}}c_1^B(M)$.
Suppose that $M$ admits a transverse Sasaki-Einstein metric. Then there are $\delta, C_\delta>0$ such that for any $K\times K$-invariant transverse K\"ahler potential $\psi$ of $\omega_g^T$ it holds
\begin{align}\label{conjecture-strong}
\mathcal K(\psi)
&\ge \delta \inf_{\tau\in Z'(T^c)} I(\psi_\tau)-C_\delta,
\end{align}
where $Z'(T^c)\subset T^c\cap {\rm Aut}^T(M)$ is a subgroup of the centre $Z(G)$ with codimension 1, ${\rm Aut}^T(M)$ is the transverse holomorphic group of $M$, $I(\cdot)$ is a functional defined by (\ref{i-functional}), and $\psi_\tau$ is an induced transverse K\"ahler potential of $\tau^*(\omega_g^T+\sqrt{-1}\partial\bar\partial \psi)$ by $\tau$.
\end{cor}

For a general Sasaki manifold which admits a transverse Sasaki-Einstein metric, we propose the following conjecture.

\begin{conj}\label{conjecture-strong-sasaki}Let $(M, g)$ be a $(2n+1)$-dimensional Sasaki manifold with $\omega_g^T={\frac{\pi}{n+1}}c_1^B(M)$. Suppose that $M$ admits a transverse Sasaki-Einstein metric. Then there
are $\delta, C_\delta>0$ such that for any $K$-invariant transverse K\"ahler potential $\psi$ of $\omega_g^T$ it holds
\begin{eqnarray}\begin{aligned}\label{conjecture-strong-2}
\mathcal K(\psi)
&\ge \delta \inf_{\tau\in Z({\rm Aut}^T(M))} I(\psi_\tau)-C_\delta,
\end{aligned}\end{eqnarray}
where $K$ is a maximal compact subgroup of ${\rm Aut}^T(M)$.
\end{conj}

Conjecture \ref{conjecture-strong-sasaki} can be regarded as a version of Tian's conjecture for $K$-invariant K\"ahler potentials in case of transverse Sasaki-Einstein manifolds \cite{T1}. Recently, Darvas and Rubinstein proved Tian's conjecture when
$Z({\rm Aut}(M))$ is replaced by ${\rm Aut}(M)$ in case of K\"ahler-Einstein manifolds \cite{DR}. We note that
Conjecture \ref{conjecture-strong-sasaki} is true by a result of Zhang \cite{Zhang}, if ${\rm Aut}^T(M)$ is finite.

An analogy of Theorem \ref{main thm} will be established for
$G$-Sasaki Ricci solitons (cf. Theorem \ref{soliton theorem}). Then by deformation of Reeb vector fields as in \cite{Martelli-Sparks-Yau, Martelli-Sparks-Yau 2}, we prove
the following openness theorem for transverse Sasaki-Ricci solitons.

\begin{theo}\label{soliton deformation}Let $\mathfrak C^\vee$ be the interior of the dual cone of $\mathfrak C$ and
$$\Sigma=\mathfrak C^\vee\cap\mathfrak a_z.$$
Set
\begin{eqnarray}\begin{aligned}\label{xi-condition}\Sigma_O=\Sigma\cap\{\xi|~\gamma_0( \xi)=-(n+1)\}.
\end{aligned}\end{eqnarray}
Suppose that a $G$-Sasaki manifold $(M,g_0)$ with the Reeb field $\xi_0$ admits a transverse Sasaki-Ricci soliton. Then for $\xi\in\Sigma_O$ sufficiently close to $\xi_0$, there is a deformation of $G$-Sasaki manifold $(M,\xi)$ with the Reeb vector field $\xi$ from $(M,g_0, \xi_0)$   such that $(M,\xi)$ admits a transverse Sasaki-Ricci soliton.
\end{theo}

It is clear that $\xi\in \Sigma_O$ in (\ref{xi-condition}) may not be rational. Thus by Theorem \ref{soliton deformation}, one can construct many irregular
$G$-Sasaki Ricci solitons from a $G$-Sasaki Einstein metric in ${\frac{\pi}{n+1}} c_1^B(M)$, see examples in  Section 8, for details.

The organization of paper is as follows. In Section 1, we  recall some basic knowledge in Sasaki geometry and in Section 2, we introduce the notion of $G$-Sasaki manifolds $M$. In Section 3, we begin to study $K\times K$-invariant metrics and discuss the moment map restricted on a torus orbit in $M$. The reduced K-energy $\mu(\cdot)$ will be computed in Section 4 and a criterion for the properness of $\mu(\cdot)$ will be given (cf. Theorem \ref{proper general class}). In Section 5, we prove Theorem \ref{proper general class} in case of $ \omega_g^T\in {\frac{\pi}{n+1}} c_1^B(M)$, and then in Section 6, we prove Theorem \ref{main thm}. Theorem \ref{soliton deformation} will be proved in Section 7. In Section 8, we discuss some examples of $G$-Sasaki Einstein metrics and $G$-Sasaki Ricci solitons.

\section{Sasaki geometry}\label{2}

By definition, a $(2n+1)$-dimensional Riemannian manifold $(M,g)$ is called a \emph{Sasaki manifold} if and only if its cone manifold $(C(M),~\bar{g})$ is a K\"ahler manifold, where
$$C(M)=M\times \mathbb R_+,\,\bar g=d\rho^2+\rho^2g,$$
and $\rho\in \mathbb{R}_+$.
Following \cite{Futaki-Ono-Wang}, we denote
\begin{eqnarray}\label{0201+}
\xi=J\rho{\frac{\partial}{\partial \rho}},~\eta(\cdot)={\frac{1}{\rho^2}}\bar g(\cdot,\xi),
\end{eqnarray}
where $J$ is the complex structure of $C(M)$. The restriction of $\xi$ on $M$ is called the \emph{Reeb vector field} of $M$. Let $\nabla^g$ be the Levi-Civita connection of $g$. Then $\Phi(X)=\nabla^g_X\xi$ defines an $(1,1)$-tensor $\Phi$ on $TM$. We call $(g,~\xi,~\eta,~\Phi)$ the \emph{Sasaki structure} of $(M,g)$.

By \cite{Boyer}, the relationship between $J$ and $\Phi$ on $M$
is given by
\begin{eqnarray}\label{J-Phi}
J(X)=\left\{\begin{aligned}
&\Phi(X)-\eta(X)\rho{\frac\partial{\partial \rho}},~& X\in TM,\\
&\xi,&X=\rho{\frac\partial{\partial \rho}}.
\end{aligned}\right.
\end{eqnarray}
Thus there are local coordinates $(x^0_{(\alpha)},z^1_{(\alpha)},...,z^n_{(\alpha)})$ on each chart $U_{(\alpha)}\subset M$, where $z^i_{(\alpha)}= x^i_{(\alpha)} +\sqrt{-1} x^{i+n}_{(\alpha)} \in\mathbb C$, $i=1,..., n$, such that
\begin{eqnarray*}
\xi=\frac{\partial}{\partial x_{(\alpha)}^0},
\end{eqnarray*}
and
$${\frac{\partial z^i_{(\beta)}}{\partial x^0_{(\alpha)}}}=0,~{\frac{\partial z^i_{(\beta)}}{\partial \bar z^j_{(\alpha)}}}=0,~\forall 1\leq i,j\leq n,$$
whenever $U_{(\alpha)}\cap U_{(\beta)}\not=\emptyset$ \cite[Chapter 6]{Boyer}. These local coordinates form a transverse holomorphic structure on $M$. Then the corresponding complex structure $\Phi^T_{(\alpha)}$ is given by
\begin{eqnarray}\begin{aligned}\label{almost-complex-structure}
\Phi^T_{(\alpha)}\left(\frac{\partial}{\partial x^{i}_{(\alpha)}}\right)=\frac{\partial}{\partial x^{i+n}_{(\alpha)}},
\end{aligned}\end{eqnarray}
which forms a global \emph{transverse complex structure} $\Phi^T$ on $M$.

Denote the transverse holomorphic group of   $M$ by ${\rm Aut}^T(M)$ and the holomorphic transformations group of $C(M)$ by ${\rm Aut}^\xi(C(M))$, which commutes  with the holomorphic flow generated by $\xi-\sqrt{-1}J\xi$. Then
$${\rm Aut}^T(M)\cong {\rm Aut}^\xi(C(M)).$$
To see this isomorphism, we note that for any $f\in {\rm Aut}^\xi(C(M))$, $f$ commutes with $\pi$, where $\pi$ is the projection from $C(M)$ to its level set $M\cong\{\rho=1\}$. Then one can define a map $\hat f$ by
\begin{eqnarray}\label{action}
\hat f=\pi_r\circ f:M\to M.
\end{eqnarray}
It is easy to see that $f_*\xi=\xi$ and ${\pi}_*\xi=\xi$. This implies that $\hat f$ preserves $\xi$. On the other hand, the complex structure $J$ on $C(M)$ is preserved by $f$.   Thus by (\ref{almost-complex-structure}), the transverse holomorphic structure $\Phi^T$ is preserved by $\hat f$.

\subsection{Basic forms and transverse K\"ahler structure}
An $m$-form $\Omega$ on $M$ is called basic if
\begin{eqnarray}\begin{aligned}\label{basic-m-form}L_{\xi}\Omega=0,~i_{\xi}\Omega=0.
\end{aligned}\end{eqnarray}
This means that
$$\Omega_{(\alpha)0i_2...i_m}dx^{i_1}_{(\alpha)}=0,~{\frac{\partial}{\partial x^0_{(\alpha)}}}\Omega_{(\alpha)i_1...i_m}=0,~\forall i_1,...,i_m, $$
if we let $\Omega=\Omega_{(\alpha)i_1...i_m}dx^{i_1}_{(\alpha)}\wedge...\wedge dx^{i_m}_{(\alpha)}$ under local transverse holomorphic coordinates $z^0_{(\alpha)},...,z^n_{(\alpha)}$.
Thus $\partial_B,~\bar\partial_B$ operators are well-defined for any basic $m$-form $\Omega$. As same as Hodge-Laplace operator, we introduce
$$\bigtriangleup_B=\sqrt{-1} (\bar \partial_B^*\bar \partial_B+\bar\partial_B\bar\partial_B^*).$$
In particular, for a basic function $f$, we have
$$\bigtriangleup_Bf=\sqrt{-1} {\rm tr}_g( \partial_B\bar \partial_B f).$$

We are interested in basic $(1,1)$-forms. Since $\bar g$ is a cone metric, we have
\begin{equation}\label{0201}
\omega_{\bar g}={\frac12}\sqrt{-1}\partial\bar\partial \rho^2.
\end{equation}
It follows
\begin{eqnarray}\begin{aligned}\label{d-eta}
d\eta=2\sqrt{-1}\partial\bar\partial\log\rho.
\end{aligned}\end{eqnarray}
This means that ${\frac12}d\eta|_{M}=\omega_g^T$ is a positive basic $(1,1)$-form. Usually, $ \omega_g^T$ is called the \emph{transverse K\"ahler form}. The following lemma shows that $\hat f\in \rm{Aut}^T(M)$ preserves transverse K\"ahler class.

\begin{lem}\label{0506+}
Let $(M,g)$ be a compact $(2n+1)$-dimensional Sasaki manifold.
Then $f^*d\eta|_M$ is a basic form for any $f\in {\rm Aut}^\xi(C(M))$. Consequently, $\hat f^*\omega_g^T \in [\omega_g^T ]_B$.
\end{lem}

\begin{proof}
Note that
\begin{eqnarray*}\begin{aligned}
f^*(d\eta)=f^*(\sqrt{-1}\partial\bar\partial\log\rho)
=\sqrt{-1}\partial\bar\partial\log f^*\rho.\notag
\end{aligned}\end{eqnarray*}
Then to prove the lemma, it suffices to show that
$\psi=\log f^*\rho-\log\rho$
satisfies
$$ \xi(\psi)(x) =0~{\rm and}~{\frac{\partial}{\partial\rho}}(\psi) (x)=0, ~\forall~x\in ~M.$$
In fact,
\begin{eqnarray}\begin{aligned}
\xi(\psi)(x)&={\frac1{f^*\rho(x)}}f_*\xi(\rho) ({f(x)})-{\frac1\rho}\xi(\rho) (x)\notag\\
&={\frac1{f^*\rho(x)}}\xi(\rho) ({f(x)})-{\frac1\rho}\xi(\rho) (x). \notag
\end{aligned}\end{eqnarray}
Then $
\xi(\psi) (x)=0$, since $\xi(\rho)\equiv0$.
On the other hand, by $\xi=J\rho{\frac{\partial}{\partial \rho}}$, we have
\begin{eqnarray*}
f_*\left(\rho{\frac{\partial}{\partial\rho}}\right)=\left(\rho{\frac{\partial}{\partial\rho}}\right).
\end{eqnarray*}
It follows
\begin{eqnarray}\begin{aligned}
{\frac{\partial}{\partial\rho}}(\psi)(x)&={\frac1{f^*\rho(x)}}\left(f_*{\frac{\partial}{\partial\rho}}\right)(\rho)({f(x)})-{\frac1{\rho(x)}}\notag\\
&={\frac1{f^*\rho(x)}}{\frac{f^*\rho(x)}{\rho(x)}}\left({\frac{\partial}{\partial\rho}}\right)\rho ({f(x)})-{\frac1{\rho(x)}}\notag\\
&=0.\notag
\end{aligned}\end{eqnarray}

\end{proof}

For a transverse K\"ahler form $ \omega_g^T$, its \emph{ transverse Ricci form} is defined by
\begin{eqnarray*}\begin{aligned}
{\rm Ric}^T(g)=-\sqrt{-1}\partial_B\bar\partial_B\log\det(g^T_{(\alpha)i\bar j}), ~{\rm on}~ U_{(\alpha)}, \notag
\end{aligned}\end{eqnarray*}
where $\omega_{ g}^T=\sqrt{-1}g_{(\alpha)i\bar j}^Tdz_{(\alpha)}^i\wedge dz^j_{(\alpha)}$ on $U_{(\alpha)}$. Clearly,
${\rm Ric}^T(g)$ is also a basic $(1,1)$-form.
Similar to the K\"ahler form, ${\rm Ric}^T(g)$ is $d_B$-closed and the basic cohomology class $[{\rm Ric}^T(g)]_B$ is independent with the choice of $\omega_g^T$ in   $[\omega_{ g}^T]_B$. We call $c_1^B(M)={\frac1{2\pi}}[{\rm Ric}^T(g)]_B$ the \emph{basic first Chern class}. It was proved in \cite[Proposition 4.3]{Futaki-Ono-Wang} that

\begin{prop}\label{C_1(D)=0} The first basic Chern class is represented by $cd\eta$ for some constant $c$ if and only if $c_1(\mathcal D)=0$, where $\mathcal D={\rm ker}(\eta)$.
\end{prop}

A Sasaki metric $g$ is called a \emph{Sasaki-Einstein metric } on $M$ if it  satisfies
$${\rm Ric}(g)=2ng.$$
In case of $c_1(\mathcal D)=0$, the above equation is equivalent to the following transverse Sasaki-Einstein equation (cf. \cite{Futaki-Ono-Wang}),
\begin{equation}\label{0205}
{\rm Ric}^T(g)=2(n+1)\omega_g^T.
\end{equation}
In particular, $c_1^B(M)>0$.

To solve (\ref{0205}), it turns to find a basic $C^{\infty}$-function
$\psi$
in the following class ( the space of  transverse K\"ahler potentials),
\begin{eqnarray*}
\mathcal H\left(\frac{1}{2}d\eta\right)=\{\psi\in C^\infty(M) \text{ is a basic function}|~ \omega_{g_\psi}^T=\frac{1}{2}d\eta+\sqrt{-1}\partial _B \bar \partial_B\psi>0\}
\end{eqnarray*}
such that
$\omega_{g_\psi}^T={\frac12}d{\eta_\psi} =   \omega_{ g}^T +\sqrt{-1}\partial _B \bar \partial_B\psi $ satisfies (\ref{0205}).
Then (\ref{0205}) is reduced to a complex Monge-Amp\`ere equation on each $U_{(\alpha)}$ with transverse holomorphic coordinates $(z^1_{(\alpha)},...,z^n_{(\alpha)})$,
\begin{align}\label{0209}
\det(g_{(\alpha)i\bar j}^T+\psi_{(\alpha),i\bar j})=\exp(-2(n+1)\psi+h)\det(g_{(\alpha)i\bar j}^T),
\end{align}
where $h$ is a basic Ricci potential  determined by
\begin{equation}\label{0206}
{\rm Ric}^T(g)=2(n+1)\omega_g^T+\sqrt{-1}\partial_B\bar{\partial}_Bh.
\end{equation}
We will discuss (\ref{0209}) in Section 6 for details.

\subsection{Futaki invariant} In general, there is no solution of (\ref{0209}) since there are some obstructions to the existence of transverse Sasaki-Einstein metrics, such as Futaki invariant (cf. \cite{Boyer-Galicki-Simanca, Futaki-Ono-Wang}). As in K\"ahler geometry, the Futaki invariant is defined for Hamiltonian vector fields. We call a complex vector field $X$ a Hamiltonian holomorphic vector field on a Sasaki manifold if $X$ satisfies (cf. \cite[Definition 4.5]{Futaki-Ono-Wang}):
\begin{itemize}
\item[(1)] On each $U_{(\alpha)}$, $\pi_{(\alpha)*}(X)$ is a (local) holomorphic vector field on $\mathbb C^n$, where
$\pi_{(\alpha)}(\cdot)$ is the projection given by
$$\pi_{(\alpha)}(x^0_{(\alpha)},z^1_{(\alpha)},...,z^n_{(\alpha)})=(z^1_{(\alpha)},...,z^n_{(\alpha)});$$
\item[(2)] The complex-valued function $U_X=\sqrt{-1}\eta(X)$ satisfies
\begin{eqnarray*}
\bar\partial_B U_X={-\frac{\sqrt{-1}}{2}}i_Xd\eta.
\end{eqnarray*}
\end{itemize}

Denote by $\mathfrak {ham}(M)$ the Lie algebra of the Hamiltonian holomorphic vector fields. The \emph{Futaki invariant} ${\rm Fut}(X)$ is defined by
\begin{eqnarray}\label{Futaki inv}
{\rm Fut}(X)=-\int_M X(h)\left({\frac12}d\eta\right)^n\wedge\eta, ~\forall ~X\in\mathfrak {ham}(M).
\end{eqnarray}
Clearly, ${\rm Fut}(X)=0$ for any $X\in \mathfrak {ham}(M)$ if $M$ admits a \emph{transverse Sasaki-Einstein metric }.
In case $c_1^B(M)>0$, it has been showed that
$\mathfrak{aut}^T(M)\cong\mathfrak {ham}(M)$,
where $\mathfrak{aut}^T(M)$ is the space of transverse holomorphic vector fields, which can be identified with the Lie algebra of ${\rm Aut}^T(M)$ (cf. \cite[Proposition 2.2]{Cho-Futaki-Ono}).

There is also a definition of Futaki invariant for general Sasaki metrics without assumption of $\omega_g^T\in{\frac{\pi}{n+1}}c_1^B(M)$. We refer the reader to \cite[Sect. 5]{Boyer-Galicki-Simanca}.

\section{Sasaki manifolds with group structure}\label{2.4}

In this section, we introduce \emph{$G$-Sasaki manifolds}.
Let $G$ be a complex, connected, reductive group of complex dimension $(n+1)$, which is the complexification of a  maximal compact subgroup $K$ as before. Assume that the centre $\mathfrak z(\mathfrak k)$ of Lie algebra of $K$ is nontrivial.

\begin{defi}\label{definition}
A \emph{$G$-Sasaki manifold} $(M,g, \xi)$ is a $(2n+1)$-dimensional Sasaki manifold with a holomorphic $G\times G$-action on $C(M)$ such that the following properties are satisfied:
\begin{itemize}
\item[(1)] There is an open and dense orbit $\mathcal O$ in $C(M)$ which is isomorphic to $G$ as a $G\times G$-homogeneous space (we will identify it with $G$);
\item[(2)] The $K\times K$-action preserves $\rho$ invariant;
\item[(3)] $\xi\in\mathfrak z(\mathfrak k).$
\end{itemize}
\end{defi}

By (\ref{0201}), the conditions (2) and (3) imply that the group $K\times K$  acts on $M$ and preserves its Sasaki structure $(g,~\xi,~\eta,~\Phi)$ invariant. Clearly, if we take $G$ an $(n+1)$-dimensional complex torus $T^c$,
then $M$ is a $(2n+1)$-dimensional toric Sasaki manifold (cf. \cite{Futaki-Ono-Wang}). We will discuss more  examples of $G$-Sasaki manifolds
 in Section 8 in the end of this paper.

Let $Z$ be the closure of $T^c$ in $C(M)$. By \cite{AB1, AB2}, $Z$ is a toric manifold. Since $\xi\in\mathfrak z(\mathfrak k)\subset\mathfrak t$, we have $\rho{\frac{\partial}{\partial \rho}}=-J\xi\in \mathfrak a$,  and so $Z$ is a K\"ahler cone over $Z\cap M$.
This implies that $Z\cap M$ is a toric Sasaki manifold with Sasaki structure $(g|_{Z\cap M},\xi|_{Z\cap M},\eta|_{Z\cap M},$
\newline $\Phi|_{Z\cap M})$ \cite{Futaki-Ono-Wang, Martelli-Sparks-Yau, Boyer}.
As in \cite{AB1, AB2, AK}, the structure of $G$-K\"ahler manifold (a polarized $G$-group compactification) is determined by its toric submanifold, the structure of $G$-Sasaki manifold $(M,g)$ will be determined by its toric Sasaki submanifold $Z\cap M$. In fact, we have

\begin{prop}\label{g-sasaki-structure}
Let $(\hat M,\omega)$ be a $K\times K$-invariant K\"ahler manifold with holomorphic $G\times G$-action which satisfies (1) in Definition \ref{definition}. Let $Z$ be the closure of $T^c$ in $\hat M$. Suppose that $(Z,\omega|_Z)$ is the K\"ahler cone over some toric Sasaki manifold $M_Z$ such that the Reeb vector field $\xi$ satisfies (3). Then $\hat M$ is a K\"ahler cone over some $G$-Sasaki manifold $M$.
\end{prop}

The proof of Proposition \ref{g-sasaki-structure} depends on the KAK-decomposition of the reductive group $G$ \cite[Sect. 7.3]{Kna}.
Let us choose a basis of right-invariant vector fields $\{E_1,...,E_{n+1}\}$ on $G\subset C(M)$ such that $\{E_1,...,E_{r+1}\}$ spans $\mathfrak t^c$, where $(r+1)$ is the dimension of $T^c$ (cf. \cite[Sect. 1]{Del2}). Denote the set of positive roots by $R_G^+$ with roots $\{\alpha_i\}_{i=1,\ldots,{\frac{n-r}2}}$. For each $\alpha=\alpha_i$, we set
$M_{\alpha}$
\[
M_{\alpha}(x) = \frac{1}{2}\left<\alpha , \nabla \psi(x)\right>
\begin{pmatrix}
\coth\alpha( x) & \sqrt{-1} \\
-\sqrt{-1} & \coth\alpha( x) \\
\end{pmatrix},~ x\in \mathfrak{a}_+,
\]
where $\mathfrak{a}_+=\{ x\in  \mathfrak{a}|~\alpha(x)>0,~\forall~\alpha\in R_G^+\}$  is the  positive Weyl chamber of $\mathfrak a$.

Let $W$ be the Weyl group of $(G, T^c)$.  The following lemma gives a formula of complex Hessian for $K\times K$-invariant functions on $G$ due to \cite{Del2}.

\begin{lem}\label{Hessian}
Any $K\times K$-invariant function $\psi$ on $G$ can descend to a $W$-invariant function (still denoted by $\psi$) on $\mathfrak a$. Moreover, there are local holomorphic coordinates on $G$ such that for $x\in \mathfrak{a}_+$,
the complex Hessian matrix of $\psi$ is diagonal by blocks as follows,
\begin{equation}\label{+21}
\mathrm{Hess}_{\mathbb{C}}(\psi)(\exp(x)) =
\begin{pmatrix}
\frac{1}{4}\mathrm{Hess}_{\mathbb{R}}(\psi)(x)& 0 & & & 0 \\
0 & M_{\alpha_1}(x) & & & 0 \\
0 & 0 & \ddots & & \vdots \\
\vdots & \vdots & & \ddots & 0\\
0 & 0 & & & M_{\alpha_p}(x)\\
\end{pmatrix}.
\end{equation}

\end{lem}

We apply Lemma \ref{Hessian} to prove Proposition \ref{g-sasaki-structure}.

\begin{proof}[Proof of Proposition \ref{g-sasaki-structure}]
Since $(Z,\omega|_Z)$ is a K\"ahler cone manifold, there is a smooth function $\rho$ on $Z$ such that
$$\omega|_Z=\frac{\sqrt{-1}}{2}\partial\bar\partial\rho^2$$
as in (\ref{0201}). Then $\rho$ can be extended to a smooth $K\times K$-invariant function on $\hat M$ (still denoted by $\rho$). Note that $\omega$ is a $K\times K$-invariant K\"ahler metric by the assumption. Thus $\omega=\frac{\sqrt{-1}}{2}\partial\bar\partial\rho^2$ on $\hat M$ (cf. \cite[Proposition 3.2)]{AK}). Let $M$ be the level set $\{\rho=1\}$ in $\hat M$. Then $M_Z=M\cap Z$. We need to prove that $\omega$ is a cone metric over $M$. By \cite[Sect. 2.1]{Martelli-Sparks-Yau 2}, it is equivalent to
show that
\begin{eqnarray}\label{cone metric}
\omega\left(\xi,X\right)=0,~\forall X\in TM.
\end{eqnarray}
It suffices to check (\ref{cone metric}) on $M\cap \mathcal O$. Note that (\ref{cone metric}) is true for any $X\in TZ$.

By the KAK-decomposition, for any $p\in M\cap \mathcal O$, there exists $k_1,k_2\in K$ such that $p'=(k_1,k_2)p\in Z$. Moreover, $p'\in M$ since $\rho(p')= \rho(p)=1$. Then by the $K\times K$-invariance of $\omega$, it holds
\begin{eqnarray*}
\omega (\xi,X)|_p=\omega (\xi,(k_1^{-1},k_2^{-1})_*X)|_{p'}, ~\forall~ X\in T_pM.
\end{eqnarray*}
We need to check (\ref{cone metric}) in the following two cases:\\
\emph{Case 1}, $(k_1^{-1},k_2^{-1})_*X\in {\rm Span}$
$\{E_{r+2},...,E_{n+1}\}$. Applying $\rho^2$ to $\psi$ in Lemma \ref{Hessian}, we see that (\ref{+21}) implies (\ref{cone metric}) since $\xi\in\mathfrak z(\mathfrak k)\subset\mathfrak t$.\\
\emph{Case 2}, $(k_1^{-1},k_2^{-1})_*X\in T_{p'}Z$. Then $(k_1^{-1},k_2^{-1})_*X$ must lie in $T_{p'}M_Z$ since $M$ is $K\times K$-invariant. Thus
$$\omega(\xi,(k_1^{-1},k_2^{-1})_*X)|_{p'}=\omega (\xi, TZ)|_{p'}= 0.$$
(\ref{cone metric}) is also true.

\end{proof}

Let $\{e^{t\xi}\}_{t\in\mathbb R}$ be the one-parameter group generated by $\xi$. We call a Sasaki manifold \emph{quasi-regular} if any orbit generated by $e^{t\xi}$ is closed. Otherwise, it is called \emph{irregular}. If the action $e^{t\xi}$ is in addition free, a quasi-regular Sasaki manifold is further called \emph{regular} (cf. \cite{Boyer, Futaki-Ono-Wang}). We note that the regularity property of $M$ is also determined by the toric Sasaki submanifold $Z\cap M$. In fact, this follows from a result of Alexeev and Brion \cite[Theorem 4.8]{AB1}: For any $p\in \hat M$, there exists $g_1,g_2\in G$ such that $p''=(g_1,g_2)p\in Z$. Then, for any $p\in M$ there is a $\rho_p\in\mathbb R$ such that $p'=e^{\rho_pJ\xi}p''\in Z\cap M$. Since both $e^{t\xi}$ and $e^{tJ\xi}$ commute with the action of $(g_1,g_2)$ by (3) of Definition \ref{definition},
\begin{eqnarray*}
e^{t\xi}p=(g_1^{-1},g_2^{-1})e^{-\rho_pJ\xi}e^{t\xi}p',~\forall t~\in\mathbb R.
\end{eqnarray*}
This means that the orbits of $p$ and $p'$ generated by $e^{t\xi}$ are isomorphic. Hence, $Z\cap M$ is regular (or quasi-regular, irregular) implies that $M$ is regular (or quasi-regular, irregular).

In the remainder  of this section, we discuss the moment map $\mu_Z$ of $(Z,\bar g|_Z)$. It is known that the image of $\mu$ is a cone minus the origin in $\mathbb R^{r+1}\cong \mathfrak a^*$ (cf. \cite{Lerman, Futaki-Ono-Wang}). Denote this cone by
\begin{equation}\label{cone def}
\mathfrak C=\bigcap_{A=1}^d\{y\in\mathfrak a^*|~ l_A(y)=u_A^iy_i\geq0\}.
\end{equation}
Without loss of  generality, we may assume that this set of $\{u_A\}$ is \emph{minimal}, which means that $\mathfrak C$ will be changed if removing any $u_A$ in (\ref{cone def}).
Since $Z\cap M$ is smooth, the cone $\mathfrak C$ is \emph{good} in sense of \cite{Lerman} (cf. \cite[Sec. 2]{Martelli-Sparks-Yau}). Namely, $\mathfrak C$ satisfies:
\begin{itemize}
\item[(C1)] Each $u_A=(u_A^1,...,u_A^{r+1})$ is a prime vector in the lattice of one-parameter groups $\mathfrak N$;
\item[(C2)] Each codimension $N$ face $\mathfrak F\subset \mathfrak C$ can be realized uniquely as the intersection of some facets $\mathfrak F_A=\{y| ~l_A(y)=0\}, A\in\{1,...,N\}\subset\{1,...,d\}$ and $$\text{Span}_{\mathbb R}\{u_1,...,u_N\}\cap\mathfrak N=\text{Span}_{\mathbb Z}\{u_1,...,u_N\}.$$
\end{itemize}

Let
\begin{eqnarray}\label{l-infty}
l_{\infty}(y)=\sum_Al_A(y).
\end{eqnarray}
Set
\begin{eqnarray}\label{0210+}
U_0^{\xi}(y)={\frac12}\sum_Al_A(y)\log l_A(y) + \frac{1}{2}l_\xi(y)\log l_\xi(y)-{\frac12}l_{\infty}(y)\log l_{\infty}(y).
\end{eqnarray}
$U_0^{\xi}(y)$ is usually called Guillimin's function on $\mathfrak C\setminus \{O\}$ \cite{Gui}. Then the Legendre function $\hat F_0$ of $U_0^{\xi}$ defined by
\begin{eqnarray}\label{0003.3+}
\hat F_0(x)=y_i\frac{\partial U_0^{\xi}}{\partial y^i}-U_0^{\xi}
\end{eqnarray}
is a K\"ahler potential on $Z$ \cite{Gui}, where
$$\frac{\partial U_0^{\xi}}{\partial y^i}=x^i:~ (y^1,..., y^{r+1})\to (x^1,...,x^{r+1})$$
is a diffeomorphism from $\mathfrak C\setminus\{ O\}$ to $\mathbb R^{r+1}$.
Conversely, for any toric cone metric with K\"ahler potential $F$ on $Z=C(Z\cap M)$, one can define a
symplectic potential $U$ of $Z$ on $\mathfrak C\setminus \{O\}$ by the Legendre transformation,
$$U(y)=x_i \frac{\partial F}{\partial x^i}-F.$$
As a version of Abreu's result for toric cone metrics, the following proposition was proved in \cite{Martelli-Sparks-Yau}.

\begin{prop}\label{sym potential}
Any symplectic potential $U$ on $Z$ associated to a K\"ahler cone metric with the Reeb vector $\xi$ can be written as
\begin{equation}\label{0211}
U=U_0^{\xi}+U',
\end{equation}
where $U'$ is a smooth homogenous function of degree $1$ on $\mathfrak C\setminus \{O\}$ such that $U$ is strictly convex.
\end{prop}

Since the cone metric $\bar g$ is $K\times K$-invariant in our case, $\mathfrak C$ is $W$-invariant \cite{AB2}. We will further assume that all $U'$ are $W$-invariant.

\section{$K\times K$-invariant metrics  in a transversely holomorphic orbit}

In this section, we reduce a $K\times K$-invariant Sasaki metric $g$ in a transversely holomorphic $n$-dimensional orbit. Let $\gamma\in\mathfrak a_z^*$ be a rational element such that $J\gamma(\xi)\not=0$. Set
$$\mathfrak k'=\ker\{J\gamma:\mathfrak k\to\mathbb R\}=\{\zeta\in \mathfrak k|~J\gamma(\zeta)=0 \}.$$
  Then $\mathfrak k'$ is a rational Lie subalgebra of $\mathfrak k$. It follows that the subgroup $K'$ generated by $\exp(\mathfrak k')$ is a closed codimension $1$ subgroup of $K$ \cite{Borel}.
Hence its complexification $H=(K')^c$ is a closed (complex) codimension $1$ reductive subgroup of $G$.
Since $\xi\in\mathfrak z(\mathfrak k)$, we see that $H\times H\subset {\rm Aut}^\xi(C(M))$.

Take a generic point $p\in M\cap\mathcal O$. Then its $H\times H$-orbit ${\rm Orb}_{C(M)}(p)$ is a complex submanifold of $C(M)$ and it is isomorphic to $H$ as a $H\times H$-homogenous space. By the isomorphism (\ref{action}), $H\times H$ can be identified with a subgroup of ${\rm Aut}^T(M)$, and ${\rm Orb}_M(p)=\pi\left({\rm Orb}_{C(M)}(p)\right)$ is its orbit of $p$ in $M$. Since $\xi\not\in \mathfrak k'$, $\xi\not\in T{\rm Orb}_M(p)$. Thus we can equip ${\rm Orb}_M(p)$ with the transverse complex structure $\Phi^T$ so that $\omega_g^T$ is a K\"ahler form on it. It can be shown that $\pi$ is a bi-holomorphic between ${\rm Orb}_{C(M)}(p)$ and ${\rm Orb}_M(p)$ (cf. \cite{Cho-Futaki-Ono, Futaki-Ono-Wang}).

We claim that $\pi$ is an isometry between $({\rm Orb}_{C(M)}(p), {\frac12}d\eta|_{{\rm Orb}_{C(M)}(p)})$ and
\newline $({\rm Orb}_M(p), \omega_g^T)$. This is because
\begin{eqnarray*}
{\pi}_*(X)=X-{\frac1{\rho^2}}\bar g\left(X,\frac{\partial}{\partial\rho}\right)\frac{\partial}{\partial\rho}, ~\forall~X\in TC(M),
\end{eqnarray*}
and by (\ref{d-eta}),
\begin{eqnarray*}
i_{\rho\frac{\partial}{\partial\rho}}d\eta
=-2\mathcal L_{\rho\frac{\partial}{\partial\rho}}(Jd\log\rho)=0.
\end{eqnarray*}
Thus for any $X, Y\in TC(M)$, we get
\begin{eqnarray}\begin{aligned}
\frac{1}{2} d\eta (X,Y)= \frac{1}{2} d\eta ({\pi_r}_*(X), {\pi_r}_*(Y))=\omega_{\bar g}({\pi_r}_*(X), {\pi_r}_*(Y))=\omega_g^T({\pi_r}_*(X), {\pi_r}_*(Y)).
\notag\end{aligned}\end{eqnarray}
This verifies the claim. Hence to study $\omega_g^T$ on $M$, it suffices to compute $\frac12d\eta$ on ${\rm Orb}_{C(M)}(p)$.

As in Section 3, we consider the closure $Z'$ of $(T')^c$-orbit ${\rm Orb}_{M}(p)$. Since $T'=\exp(\mathfrak t')$ is a maximal compact torus of $K'$, $Z'$ is just the torus orbit corresponding to $(T')^c$ in $Z$. By (\ref{d-eta}), we see that
$$\omega_g^T(p)=\sqrt{-1}\partial\overline \partial\log \rho (p). $$
Then by the above claim, we get
\begin{align}\label{potential-P}\omega_g^T|_{ (T')^c \subset Z'}=\sqrt{-1}\partial\overline \partial\log \rho( (T')^c(p)).
\end{align}
It implies that $\varphi((T')^c) =\log \rho|_{ (T')^c(p) }$ is a K\"ahler potential of the restriction of $\omega_g^T$ on the orbit $(T')^c \subset Z'$. Thus $\log \rho$ can be regarded as a convex function in $\mathbb R^r$ since $\omega_g^T|_{Z'}$ is $K'$-invariant. We shall compute the polytope of moment $\mu'$ associated to $\omega_g^T|_{Z'}$ with the action $(T')^c$ below.

\subsection{Moment polytope of $(Z', \omega_g^T|_{Z'})$}\label{4.2}

Let $\mu_Z$ be the moment map of $(Z,\bar g|_Z)$ as in Section 2.
Since $\rho=1$ on $M$, by a direct computation, the image of $Z\cap M$ under $\mu_Z$ is an intersection of $\mathfrak C$ with the \emph{characteristic hyperplane} $\{y| ~l_{\xi}(y)=\xi^iy_i=1\}$, which is a polytope $\mathcal P$ in $\mathfrak a^*$,
\begin{eqnarray}\label{eq P}
\mathcal P=\{y|~ l_\xi(y)=1, u_A^iy_i\geq0,~\forall A\}.
\end{eqnarray}
Thus $\mathfrak C$ is a cone over it. Since $M$ is compact, $\mathcal P$ must be bounded. Hence $\xi$ lies in the interior of the dual cone of $\mathfrak C$.

Let $\mathfrak a'=J\mathfrak k'$. Let $\iota:\mathfrak a'\to\mathfrak a$ be the inclusion and $\iota^*:\mathfrak a^*\to ({\mathfrak a'})^*$  its dual map.
Then $\gamma (\mathfrak a')=0$. It follows that
\begin{eqnarray}\label{0305}
\iota^*(y)=y-{\frac{\langle y, \gamma\rangle  }{\langle\gamma,\gamma\rangle}}\gamma,~\forall ~y\in\mathfrak a^*.
\end{eqnarray}
Thus we can identify $ ({\mathfrak a'})^*$ with the image of $\iota^*$ in $\mathfrak a^*$, which is a codimension 1 subspace orthogonal to $\gamma$. Let $P$ be the image of $\mu'$ in $({\mathfrak a'})^*$. The following proposition shows that $P$ is equal to $\iota^*(\mathcal P)$, and consequently $P$ depends only on the choice of $\gamma$.

\begin{prop}\label{polytope P}
$P$ is equal to $\iota^*(\mathcal P)$, which is a bounded, convex and $W$-invariant polytope. More precisely,
\begin{eqnarray}\label{P def}
P=\{v\in(\mathfrak a' )^*|~l'_A(v)\geq0,~A=1,...,d\},
\end{eqnarray}
where
\begin{eqnarray}\label{4.4+}
l'_A(v)=\left(u_A^i-{ \frac{\gamma(u_A)}{\gamma(\xi)} }\xi^i\right)v_i+{\frac{\gamma(u_A)}{\gamma(\xi)}}.
\end{eqnarray}
Furthermore, each codimension $N$ face of $P$ is exactly intersections of $N$ facets. In particular, each vertex of $P$ is exactly the intersection of $r$ facets.
\end{prop}

\begin{proof}
Note that the inclusion $\iota:(\mathfrak t')^c\to\mathfrak t^c$ of Lie algebras induces a holomorphic embedding (still denoted by $\iota$) of toric manifold $Z'$ into $Z$.
Then for any holomorphic vector field $X$ on $Z'$, by (\ref{potential-P}), we have
\begin{eqnarray*}
i_{X}\left(\iota^*\frac12d\eta\right)&=&\iota^*\left(i_{\iota_*X}\left(\frac12d\eta\right)\right)\\
&=&{\sqrt{-1}}\iota^*(\bar\partial(\iota_*X(\log \rho))).
\end{eqnarray*}
Thus, by the definition of moment map, 
it follows that
\begin{eqnarray*}
\mu'=\iota^*\left(\nabla\log\rho\right).
\end{eqnarray*}
On the other hand, by Proposition \ref{sym potential}, $U'$ is homogenous of degree $1$. Then
\begin{eqnarray}\label{0303+}
\begin{aligned}
2y_k{\frac{\partial x^i}{\partial y_k}}&=2y_k{\frac{\partial^2 U}{\partial y_k\partial y_i}}\notag\\
&=2y_k{\frac{\partial^2 U_0^{\xi}}{\partial y_k\partial y_i}}\notag\\
&=\xi^i.
\end{aligned}
\end{eqnarray}
Thus
\begin{eqnarray}\label{0306+}
l_{\xi}(y)=\xi^iy_i=\rho^2,
\end{eqnarray}
and so
$${\frac12}d\eta={\frac{\sqrt{-1}}{2}}\partial\bar\partial\log l_{\xi}(y). $$
As a consequence,
\begin{eqnarray*}
\mu'={\frac{\iota^*(y)}{l_{\xi}(y)}}.
\end{eqnarray*}
This means that
\begin{eqnarray}\label{0307}
P=\{v=\iota^*(y)|~y\in\mathfrak C,~l_{\xi}(y)=1\},
\end{eqnarray}
which is equivalent to $\iota^*(\mathcal P)$. In particular, $P$ is bounded, convex and $W$-invariant.

By (\ref{eq P}) and (\ref{0305}), it is easy to see that the inverse of $ \iota$ is given by
\begin{eqnarray}\label{0306}
\iota^{*-1}(v)=v+{\frac{1-v(\xi)}{\gamma(\xi)}}\gamma,~\forall v~\in\mathfrak {a'}^*.
\end{eqnarray}
Thus by (\ref{0307}), we obtain (\ref{P def}) immediately.

The second part in the proposition follows from the property of $\mathcal P$. In fact, according to (\ref{0307}), any codimension $N$ face of $\mathcal P$ is an intersection of a codimension $N$ face of $\mathfrak C$ with the {characteristic hyperplane} $\{l_{\xi}=1\}$. Then by the property (C2) of $\mathfrak C $ in Section \ref{2.4}, each codimension $N$ face of $\mathcal P$ is exactly intersections of $N$ facets.
\end{proof}

\subsection{Space of Legendre functions}

In this subsection, we determine the space of Legendre functions on $P$ associated to $K\times K$-invariant transverse  K\"ahler potentials of $\omega^T_\psi\in[\omega_g^T]_B$. 
For convenience, we set the class of $K\times K$-invariant K\"ahler potentials of $(M, \frac{1}{2}d\eta)$ by $\mathcal H_{K\times K}\left(\frac{1}{2}\eta\right)$. Namely,
\begin{eqnarray*}
\mathcal H_{K\times K}\left(\frac{1}{2}d\eta\right)=\{\psi~ |\psi\in \mathcal H\left(\frac{1}{2}d\eta\right)| ~\psi~\text{ is }K\times K\text{-invariant}\}.
\end{eqnarray*}

Let $\hat F_0$ be the K\"ahler potential associated to the symplectic potential $U_0^{\xi}$ on $Z$ in (\ref{0003.3+}). Then
by Proposition \ref{g-sasaki-structure}, $\hat F_0$ extends to a function $\frac{1}{2}\rho_0^2$ on $C(M)$ which induces a
$G$-Sasaki manifold $(M,\frac{1}{2}d\eta_0)$ with
$$\frac{1}{2}d\eta_0=\sqrt{-1} \partial\bar\partial \log\rho_0.$$
By (\ref{potential-P}), $\frac{1}{2}\log\hat F_0$ is a K\"ahler potential on $Z'$. Let
$u_0$ be its Legendre function. Then
\begin{eqnarray}\label{0308}
u_0(x)&=&{\frac12}\left((\log \hat F_0)_{,i}x^i-\log \hat F_0\right)\notag\\
&=&{\frac12}\left({\frac{\hat F_{0,i}x^i}{\hat F_0}}-\log \hat F_0\right)\notag\\
&=&{\frac12}\left({\frac{U_0^{\xi}(\nabla \hat F_0)}{\hat F_0}}-\log \hat F_0+1\right).
\end{eqnarray}

On the other hand, by (\ref{0306+}), we have
$$2\hat F_0=l_\xi(y),$$
where $y=\nabla \hat F_0.$
Then by (\ref{0210+}), we get
\begin{eqnarray}\begin{aligned}
&{\frac{U_0^{\xi}(\nabla \hat F_0)}{\hat F_0}}-\log \hat F_0\notag\\
&=\sum_A{\frac{l_A(y)}{l_\xi(y)}}\log l_A(y)+\log l_{\xi}(y)-{\frac{l_\infty(y)}{l_\xi(y)}}\log l_{\infty}(y)\notag\\
&-\log l_{\xi}(y)+\log2\notag\\
&=\sum_A{\frac{l_A(y)}{l_\xi(y)}}\log {\frac{l_A(y)}{l_\xi(y)}}
+\sum_A{\frac{l_A(y)}{l_\xi(y)}}\log l_\xi(y)\notag\\
&-{\frac{l_\infty(y)}{l_\xi(y)}}\log {\frac{l_\infty(y)}{l_\xi(y)}}
-{\frac{l_\infty(y)}{l_\xi(y)}}\log{l_\xi(y)}+\log2.\notag
\end{aligned}\end{eqnarray}
Hence, by (\ref{l-infty}), it follows that
\begin{align}\label{f-0-u}
&{\frac{U_0^{\xi}(\nabla \hat F_0)}{\hat F_0}}-\log \hat F_0\notag\\
&
=\sum_A{\frac{l_A(y)}{l_\xi(y)}}\log {\frac{l_A(y)}{l_\xi(y)}}
-{\frac{l_\infty(y)}{l_\xi(y)}}\log {\frac{l_\infty(y)}{l_\xi(y)}}+\log2.
\end{align}

Let
\begin{eqnarray*}
v={\frac{\iota^*(\nabla \hat F_0)}{2\hat F_0}}
={\frac{\iota^*(y)}{\xi^iy_i}}.
\end{eqnarray*}
Then by (\ref{0306}),
\begin{eqnarray*}
{\frac{y}{\xi^iy_i}}
=v+{\frac{1-v(\xi)}{\gamma(\xi)}}\gamma.
\end{eqnarray*}
Thus by (\ref{4.4+}), we see that
\begin{eqnarray*}
{\frac{l_A(y)}{l_\xi(y)}}
=l'_A(v),
\end{eqnarray*}
and
\begin{eqnarray*}
{\frac{l_\infty(y)}{l_\xi(y)}}=\sum_Al_A'(v)=l'_\infty(v).
\end{eqnarray*}
Plugging (\ref{f-0-u}) and the above two equalities into (\ref{0308}), we derive
\begin{eqnarray}\label{u-0-function}
&&u_0(v)={\frac12}\sum_Al'_A(v)\log l'_A(v)-{\frac12}l'_\infty(v)\log l'_\infty(v)+\log2+{\frac12}.
\end{eqnarray}
Note that
$l'_\infty(v)$ has strictly positive lower bound on $\overline P$. Then
\begin{eqnarray*}
l_\infty'(v)\log l'_\infty(v)\in C^\infty(\overline P).
\end{eqnarray*}
Set
\begin{eqnarray}\label{u-g}
u_G={\frac12}\sum_Al'_A(v)\log l'_A(v).
\end{eqnarray}
We see that $u_0-u_G\in C^{\infty}(\overline P)$.

Set
\begin{eqnarray*}
\mathcal C_{P,W}=\{u|~u-u_G\in C^{\infty}(\overline{P}), u\text{ is strictly convex and $W$-invariant}\}.
\end{eqnarray*}
We prove

\begin{lem}\label{Legendre functions}
Let $\psi\in\mathcal H_{K\times K}\left(\frac{1}{2}d\eta\right)$ and $\varphi_\psi$ be the K\"ahler potential of $\omega_\psi^T\in[\omega_g^T]_B$.
Then the Legendre function $u_\psi$ of $\varphi_\psi$ belongs to $C_{P,W}$.
\end{lem}

\begin{proof}
Without loss of  generality, we may assume (cf. \cite[Proposition 4.2]{Futaki-Ono-Wang}),
\begin{eqnarray*}
F_\psi=e^{2\psi}\hat F_0.
\end{eqnarray*}
Then $\varphi_\psi={\frac12}\log \hat F_0+\psi$ is a toric K\"ahler potential of $\omega_\psi^T$. Since $\psi$ is basic,
$\xi(\psi)=0$.
Thus we get
\begin{eqnarray}\label{0306++}
\xi^i(e^{2\psi}\hat F_0)_{,i}=e^{2\psi}\xi^i\hat F_{0,i}
=2F_\psi.
\end{eqnarray}
Hence, if we set $y=\nabla F_\psi$, then
\begin{eqnarray*}
\xi^iy_i=2F_\psi.
\end{eqnarray*}

On the other hand, by (\ref{0211}), the Legendre function $U_\psi$ of $F_\psi$ can be written as
\begin{eqnarray*}
U_\psi=U_0^{\xi}+U'
\end{eqnarray*}
for some smooth, homogenous degree 1 function $U'$ on $\mathfrak C$. Then analogous to (\ref{0308}), the Legendre function of $\varphi_\psi={\frac12}\log F_\psi|_{\mathfrak a'}$, is given by
\begin{eqnarray*}
u_\psi(x)
&=&{\frac12}\left({\frac{U_\psi(\nabla F_\psi)}{F_\psi}}-\log F_\psi+1\right)\notag\\
&=&{\frac12}\left({\frac{U_0^{\xi}(y)}{l_\xi(y)}}-\log l_\xi(y)+\log2+1+U'\left(\frac{y}{l_\xi(y)}\right)\right).
\end{eqnarray*}
Similarly as in the proof of (\ref{u-0-function}), for $v=\iota^*\left(\frac{y}{l_\xi(y)}\right)$, we get
\begin{eqnarray*}
u_\psi(v)=u_0(v)+{\frac12}U'(v).
\end{eqnarray*}
The lemma then follows from the fact that $u_0-u_G\in C^{\infty}(\overline P)$.

The convexity and $W$-invariance of $\varphi_\psi$ follows exactly as in the K\"ahler case (cf. \cite{Del2, LZZ}).

\end{proof}

\section{The reduced K-energy $\mu(\cdot)$}

In this section, we compute the K-energy $\mathcal K(\cdot)$ on $\mathcal H_{K\times K}\left(\frac{1}{2}d\eta\right)$ on a Sasaki manifold $(M, \frac{1}{2}d\eta)$ in terms of Legendre functions on $P$ as in \cite{LZZ}. Recall that the average   $\bar S^T$ of transverse scalar curvature
  $S^T$ of $\frac{1}{2}d\eta$ is given by
\begin{eqnarray*}
\bar S^T={\frac{1}{V}}\int_MS^T( \frac{1}{2} d\eta)^n\wedge \eta,
\end{eqnarray*}
where
\begin{eqnarray*}
V=\int_M \left(\frac{1}{2} d\eta \right)^n\wedge \eta
\end{eqnarray*}
is the volume of $\frac{1}{2} d\eta=\omega_g^T$.
As same as $V$, $\bar S^T$ is independent of the choice of $ \frac{1}{2}d\eta_{\psi}$ with $\psi\in \mathcal H\left(\frac{1}{2}d\eta\right)$.
The K-energy on $(M, \frac{1}{2} d\eta)$ is introduced by Futaki-Ono-Wang as follows \cite{Futaki-Ono-Wang},
\begin{eqnarray}\label{K-energy}
\mathcal K(\psi)=-{\frac1V}\int_0^1\int_M\dot\psi_t(S_t^T-\bar S^T)(\frac{1}{2}d\eta_{\psi_t})^n\wedge \eta_{\psi_t}\wedge dt,
 \end{eqnarray}
for any  $\psi\in\mathcal H( \frac{1}{2} d\eta)$,  where $\{\psi_t\}_{t\in[0,1]}$ is any smooth path in $\mathcal H\left(\frac{1}{2}d\eta\right)$ joining $0$ and $\psi$, and $S^T_t$
is the transverse scalar curvature of $\frac{1}{2} d\eta_{\psi_t}$.

Let $dh$ be a Haar measure of $H$. Write the complex Monge-Amp\`ere operator measure on ${\rm Orb}_{M}(p)$, induced by $   \frac{1}{2} d\eta_\psi$ as
$$(\sqrt{-1}\partial\bar{\partial}\varphi_\psi)^n={\rm MA}_{\mathbb C}(\varphi_\psi)dh.$$
Note that $H$ and $G$ have the same roots system. Then
by Lemma \ref{Hessian}, we have
\begin{equation}\label{MA}
\mathrm{MA}_{\mathbb{C}}(\varphi_\psi)(\exp(x))= \frac{1}{4^{r+p}}
\mathrm{MA}_{\mathbb{R}}(\varphi_\psi)(x)\frac{1}{\mathbf J(x)}\prod_{\alpha \in R_G^+} \left<\alpha,\nabla \varphi_\psi(x)\right>^2,~\forall x\in\mathfrak a_+',
\end{equation}
where $\mathbf J(x)=\prod_{\alpha\in R_G^+}\sinh^2\alpha(x)$.

The following lemma gives a version of KAK-integration formula on a $G$-Sasak manifold.

\begin{lem}\label{KAK int}
Let $(M, \frac{1}{2} d\eta)$ be a $G$-Sasaki manifold with Reeb vector field $\xi$. Then there is a constant $C_0$ which depends only on $\xi$ and $H$ such that for any $K\times K$-invariant function $f$,
\begin{eqnarray}\label{kka-formula-sasaki}
\int_Mf(d\eta)^n\wedge\eta=C_0\int_{\mathfrak a_{+}'}f\mathbf \rm{MA_{\mathbb R}}(\varphi_0)\prod_{\alpha\in R_G^+}\langle\alpha,\nabla\varphi_0(x)\rangle^2dx, 
\end{eqnarray}
where $\varphi_0$ is  a transverse K\"ahler  potential of  $ \frac{1}{2} d\eta=\sqrt{-1}\partial\bar\partial\varphi_0$.
\end{lem}

\begin{proof}
It suffices to do the integration on the open dense orbit $M\cap\mathcal O$. We claim that for any $q\in M\cap\mathcal O$, the flow line generated by $\xi$ through $q$ intersects $Orb_M(p)$. In fact, by using KAK-decomposition, we may assume $q\in Z\cap\mathcal O$ without loss of generality. The claim then follows from \cite[Proposition 7.2]{Futaki-Ono-Wang}.

Note that all $e^{t\xi}$-orbits in $M\cap\mathcal O$ are isomorphic to each other.  Then we have  two cases.\\
\emph{Case1.} The $e^{t\xi}$-orbits in $M\cap\mathcal O$ are all compact, so they can be parameterized by $S^1$. In this case, the integration can be taken first along each $e^{t\xi}$-orbit and then over $Orb_M(p)$. On the other hand, in the coordinates chosen in Section 1,
\begin{eqnarray*}
(\frac{1}{2}d\eta)^n\wedge\eta=(\omega_g^T)^n\wedge dx_{(\alpha)}^0.
\end{eqnarray*}
Since $f$ is $e^{t\xi}$-invariant, we have $f=f(z_{(\alpha)})$, which is a constant along each $e^{t\xi}$-orbit. Thus
\begin{eqnarray*}
\int_Mf(\frac{1}{2} d\eta)^n\wedge\eta=C_0\int_{Orb_M(p)}f(\omega_g^T|_{Orb_M(p)})^n,
\end{eqnarray*}
where $C_0$ is a constant independent of $f$.  By  (\ref{MA}),  we get (\ref{kka-formula-sasaki}).\\
\emph{Case 2.} The $e^{t\xi}$-orbits in $M\cap\mathcal O$ are non-compact. In this case, let $T_\xi$ be the closure of $e^{t\xi}$. It is a compact torus in $Z(K)$ whose dimension is at least $2$. Thus $\mathfrak t_\xi\cap\mathfrak k'\not=\emptyset$. Take an $\varsigma'\in\mathfrak t_\xi\cap\mathfrak k'$ such that $\xi'=\xi+\varsigma'$ generates a compact group and let $\theta'$ be the dual of $\xi'$. Then
\begin{eqnarray*}
(\frac{1}{2}d\eta)^n\wedge\eta=(\omega_g^T)^n\wedge\theta'.
\end{eqnarray*}
Since $f$ is also $e^{t\xi'}$-invariant,  (\ref{kka-formula-sasaki}) follows from the proof in
Case 1.
\end{proof}

For any $u\in\mathcal C_{P,W}$, we denote
\begin{eqnarray*}
u_{,i}={\frac{\partial u}{\partial v_i}},~u_{,ij}={\frac{\partial^2 u}{\partial v_i\partial v_j}}
\end{eqnarray*}
and $(u^{ij})$ the inverse matrix of $(u_{,ij})$. By Proposition \ref{polytope P}, near any point $p\in\partial P$, there exists local adapt coordinates introduced by \cite{D2}. That is, for any $p\in\partial P$, we can choose affine coordinates $\{v_i\}_{i=1,...,r}$ on $\mathbb R^r$ such that a neighbourhood of $p$ in $P$ is given by
$$v_1,...,v_N\geq0$$
for some $1\leq N\leq r$. Thus by (\ref{P def}) and \cite[Proposition 2]{D2} we have for any $u\in\mathcal C_{P,W}$,
\begin{eqnarray}\label{boundary behavior of u^{ij}}
u^{ij}\nu_{Ai} \to 0~\text{ and }~u^{ij}_{,j}\nu_{Ai} \to -\frac{2}{\lambda_A}\langle v, \nu_{ A}\rangle,
\end{eqnarray}
as $v$ goes to a facet $\mathfrak F'_A=\{v  \in(\mathfrak a' )^*|~l'_A(v)=0\}$ of $P_+=P\cap \mathfrak a_+^*$. Here
\begin{eqnarray*}
\lambda_A=\frac{\gamma( u_A)}{\gamma(\xi)},
\end{eqnarray*}
and $\nu_A$ denotes the unit outer normal vector of $\mathfrak F'_A$ .

Let
\begin{eqnarray}\label{constant-a}
\Lambda_A={\frac2{\lambda_A}}(1-2\sigma (u_A)), ~\forall~\mathfrak F_A\cap  \mathfrak a_+^*\neq \emptyset.
\end{eqnarray}
For any $u\in\mathcal C_{P,W}$, we define a functional $\mu(\cdot)$ by
\begin{align}\label{reduced-energy-formula}
&\mu(u) \notag\\
&={\frac{1}{V_P}}\sum_{\{A|~ \mathfrak F_A\cap  \mathfrak a_+^*\neq \emptyset\}}    \Lambda_A\int_{\mathfrak F'_A\cap P_+}u\left\langle v,\nu_A\right\rangle\pi\,d\sigma_0  \notag\\
&-\frac{\bar S}{V_P}\int_{P_+}u\pi\,dv-{\frac{4}{V_P}}\int_{P_+}\sigma(\nabla u)\pi\,dv\notag\\
&-{\frac{1}{V_P}}\int_{P_+}\log\det\left(u_{ij}\right)\pi\,dv
+{\frac{1}{V_P}}\int_{P_+}[\chi\left(\nabla u \right)+4\sigma(\nabla u) ]\pi\,dv,
\end{align}
where $\pi(v)=\prod_{\alpha\in R_G^+}\langle\alpha,v\rangle^2,\,\chi(x)=-\log\mathbf J(x)$ and $V_P=\int_{P_+}\pi\,dv$. 
The following proposition shows that K-energy is same to the functional $\mu(\cdot)$.

\begin{prop}\label{mu(u)}
\begin{align}\label{reduced-energy}
\mathcal K(\psi)=\mu(u_\psi) +const., ~\forall \psi\in\mathcal H_{K\times K}\left(\frac{1}{2} d\eta\right).
\end{align}

\end{prop}

\begin{proof}
By Lemma \ref{KAK int}, we have
\begin{eqnarray}\label{0501+}
\mathcal K(\psi)=-{\frac {C_0}V}\int_0^1\int_{(\mathfrak a')_{+}}\dot\psi_t(S^T_t-\bar S^T)(\omega^T_t)^n\wedge dt.
\end{eqnarray}
On the other hand,
analogous to \cite[Lemma 2.4]{LZZ}, we see that
\begin{align}\label{+S}
S^T_t=&-u^{ij}_{t,ij}-2u^{ij}_{t,j}{\frac{\pi_{,i}}{\pi}}-u_t^{ij}{\frac{\pi_{,ij}}{\pi}} \notag\\
&-u_{t,ik}
\left.{\frac{\partial^2 \chi}{\partial x^i\partial x^k}}\right|_{x=\nabla{u_t}}-\left.{\frac{\partial \chi}{\partial x^i}}\right|_{x=\nabla{u_t}}{\frac{\pi_{,i}}{\pi}}.
\end{align}
Consequently,
\begin{align}\label{average-r}\bar S^T={\frac 1{V_P}}\sum_{\{A|~ \mathfrak F_A\cap  \mathfrak a_+^*\neq \emptyset\}}   \Lambda_A\int_{\mathfrak F'_A\cap P_+}\langle v,\nu_A\rangle\pi~d\sigma_0.
\end{align}
Then substituting (\ref{+S}) and (\ref{average-r}) into (\ref{0501+}), and taking integration by parts together with (\ref{boundary behavior of u^{ij}}), we get
\begin{align} &{V_P}\cdot\mathcal K(\psi)\notag\\
&=\sum_{\{A|~ \mathfrak F_A\cap  \mathfrak a_+^*\neq \emptyset\}} {\Lambda_A}\int_{\mathfrak F'_A\cap P_+}u_\psi\left\langle v,\nu_A\right\rangle\pi\,d\sigma_0-\bar S\cdot\int_{P_+}u_\psi\pi\,dv\notag\\
&-\int_{P_+}\log\det\left(u_{\psi,ij}\right)\pi\,dv+\int_{P_+}[\chi\left(\nabla u_\psi\right)+4\sigma(\nabla u_\psi) ]\pi\,dv+const.\notag
\end{align}
Note that
\begin{eqnarray*}
V=\int_M (d\eta)^n\wedge\eta=C_0\cdot V_P.
\end{eqnarray*}
Thus (\ref{reduced-energy}) is true.
A detailed proof can be found in \cite[Proposition 3.1]{LZZ}.
\end{proof}

We call $\mu(\cdot)$ the \emph{reduced K-energy} of $\mathcal K(\cdot)$ as in \cite{D1, ZZ, LZZ}. By Proposition \ref{mu(u)}, $\mu(\cdot)$ is well-defined on $\mathcal C_{P,W}$. Note that the nonlinear part
$$-{\frac{1}{V_P}}\int_{P_+}\log\det\left(u_{ij}\right)\pi\,dv+{\frac{1}{V_P}}\int_{P_+}[\chi\left(\nabla u \right)+4\sigma(\nabla u)]\pi\,dv$$
is invariant by adding a linear function whcih depends only on $\mathfrak a'_z=\mathfrak a'\cap\mathfrak z(\mathfrak h)$.
We will use the Futaki invariant to normalize $u$ in $\mathcal C_{P,W}$.
By Proposition \ref{mu(u)}, we observe

\begin{lem}\label{Fut=0 and L=0}Let
\begin{align}
\mathcal L(u) &={\frac{1}{V_P}}\sum_{\{A|~ \mathfrak F_A\cap  \mathfrak a_+^*\neq \emptyset\}} \Lambda_A\int_{\mathfrak F'_A\cap P_+}u\left\langle v,\nu_A\right\rangle\pi\,d\sigma_0  \notag\\
&-\frac{\bar S}{V_P}\int_{P_+}u\pi\,dv-{\frac{4}{V_P}}\int_{P_+}\sigma(\nabla u)\pi\,dv.\notag
\end{align}
Then $M$ has vanishing Futaki invariant if and only if
\begin{eqnarray}\label{0507}
\mathcal L(a^iv_i)=0
\end{eqnarray}
for any $a=(a^i)$ in $\mathfrak a'_z$.
\end{lem}

\begin{proof}Let $\sigma_X(t)$ be a one parameter subgroup of ${\rm Aut}^T(M)$ generated by some $X=\sum a^iE_i^0 \in\mathfrak a_z'$.
Then by Lemma \ref{0506+}, we have
\begin{eqnarray*}
[\sigma_X(t)]^*\omega_g^T=\omega_g^T+\sqrt{-1}\partial\bar\partial \phi_t
\end{eqnarray*}
for some basic function $\phi_t$. By \cite[Proposition 5.2]{Boyer-Galicki-Simanca}, it follows
\begin{eqnarray*}
{\frac{d}{dt}}\mathcal K(\psi_t)=-\frac{1}{V}{\rm Fut}(X).
\end{eqnarray*}
On the other hand, by Proposition \ref{mu(u)}, as in \cite{LZZ}, we see that
$${\frac{d}{dt}}\mathcal K(\psi_t)=\mathcal L(a^iv_i+c)=\mathcal L(a^iv_i),$$
where $c$ is some constant. Combining the above two relations, we prove the lemma.
\end{proof}

Without loss of  generality, we may choose $\gamma$ such that $O\in P$. When the Futaki invariant vanishes, ${\mathcal C}_{P,W}$ can be normalized by a set
$$\hat{\mathcal C}_{P,W}=\{u\in\mathcal C_{P,W}|~u\geq u(O)=0\}.$$
In fact, we have

\begin{lem}\label{normalization of u} Assume that $Fut(\cdot)=0$. Then for any $\psi\in \mathcal H_{K\times K}(\frac{1}{2}d \eta)$, there is $\sigma\in Z(H)$ such that
the Legendre function $\hat u$ of $\mathbf \varphi_{\sigma}$ belongs to $\hat{\mathcal C}_{P,W}.$
\end{lem}

\begin{proof}Let $u$ be the Legendre function of $\varphi_{\psi}$. Then
$u\in{\mathcal C}_{P,W}$. By the $W$-invariance, $a=\nabla u(O)\in\mathfrak a'_z$. Let $\sigma_t^a$ be the one parameter subgroup of $H^c$ generated by $-a$. By Lemma \ref{0506+}, there is a $\psi_\sigma\in\mathcal H_{K\times K}\left(\frac{1}{2}d\eta\right)$ with $(\varphi_0+\psi_\sigma)(O)=0$ such that
\begin{eqnarray*}
(\sigma_1^a)^*{\frac12}d\eta_\psi={\frac12}d\eta+\sqrt{-1}\partial_B\bar\partial_B\psi_\sigma,
\end{eqnarray*}
where $\sigma=\sigma_1^a$.
Then one can check that the Legendre function $\hat u$ of $\varphi_0+\psi_\sigma$ is given by
$$\hat u= u-a^iv_i -u(O).$$ Thus $\hat u\in\hat{\mathcal C}_{P,W}.$
\end{proof}

\subsection{A criterion for the properness of  K-energy}
Recall $I$-functional,
\begin{align}\label{i-functional}
I(\psi)=I(\omega_g^T, \psi)={\frac1V}\int_M\psi[(\frac{1}{2}d\eta)^n\wedge \eta-(\frac{1}{2}d\eta_{\psi})^n\wedge \eta_{\psi}],
\end{align}
where  $\psi \in\mathcal H\left(\frac{1}{2}d\eta\right)$.
We call $\mathcal K(\cdot)$ \emph{proper on $\mathcal H\left(\frac{1}{2}d\eta\right)$} if there is an increasing function $f(t):\mathbb R_{\geq0}\to\mathbb R$ which satisfies $\lim_{t\to+\infty}f(t)=+\infty$ such that
\begin{align}\label{f-function}
\mathcal K(\psi)\geq f(I(\psi)),~\forall\psi\in\mathcal H\left(\frac{1}{2}d\eta\right).
\end{align}

In view of Lemma \ref{0506+}, the action of ${\rm Aut}^T(M)$ on $M$ preserves $[\omega_g^T]_B$. We introduce

\begin{defi}\label{properness-definition}Let  $K$ be a   maximal compact
subgroup  of  ${\rm Aut}^T(M)$  and  $\mathcal H_{K}\left(\frac{1}{2}d\eta\right)$  the subset of $K$-invariant Sasaki metrics in $\mathcal H (\frac{1}{2}d \eta)$.    Let $G_0$ be another   subgroup of  ${\rm Aut}^T(M)$.
$\mathcal K(\cdot)$ is called \emph{proper on $\mathcal H_{K}\left(\frac{1}{2}d\eta\right)$ modulo $G_0$} if there is a $f$ as in (\ref{f-function}) such that
\begin{eqnarray*}
\mathcal K(\psi)\geq \inf_{\sigma\in G_0}f(I(\psi_\sigma)),~\forall\psi\in\mathcal H_{K}\left(\frac{1}{2}d\eta\right),
\end{eqnarray*}
where $\psi_\sigma$ is defined by $\frac{1}{2}d\eta_{\psi_\sigma}=\frac{1}{2}\sigma^*(d \eta_\psi)={\frac12}d\eta+\sqrt{-1}\partial_B\bar\partial_B\psi_\sigma$.
\end{defi}

Let $bar$ and $\widetilde{bar}$ be the weighted barycenters of $P_+$ and $\partial P_+$, respectively, which are defined by
\begin{eqnarray}\begin{aligned}
bar&=\frac{\int_{P_+}v \pi\,dv}{\int_{ P_+} \pi\,dv},\notag\\
\widetilde{bar}&=\frac{\sum_{\{A|~ \mathfrak F_A\cap  \mathfrak a_+^*\neq \emptyset\}} \Lambda_A\int_{\mathfrak F'_A\cap P_+}v\langle v,\nu_A\rangle\pi\,d\sigma_0}{\sum_{\{A|~ \mathfrak F_A\cap  \mathfrak a_+^*\neq \emptyset\}} \Lambda_A\int_{\mathfrak F'_A\cap P_+}\langle v,\nu_A\rangle\pi\,d\sigma_0}.\notag
\end{aligned}\end{eqnarray}
Let $bar_{ss}$ and $\widetilde{bar}_{ss}$ are projections of $bar$ and $\widetilde{bar}$ to the semi-simple part $\mathfrak a_{ss}^*$ in $\mathfrak a^*$, respectively. Then following the argument in the proof of main theorem in \cite[Theorem 1.2]{LZZ}, we have

\begin{theo}\label{proper general class}
Let $(M,g)$ be a compact $G$-Sasaki manifold with vanishing Futaki invariant. Suppose that there is a $\gamma\in\mathfrak a_z^*$ such that $\gamma (\xi)\not=0$ and
\begin{eqnarray}
&&\left(\min_A\Lambda_A\cdot\widetilde{bar}_{ss}-4\sigma\right)\in\Xi,   ~\forall ~\mathfrak F_A\cap  \mathfrak a_+^*\neq \emptyset, \label{tildebar1} \\
&&\left(\widetilde{bar}_{ss}-bar_{ss}\right)\in\bar{\Xi}, \label{tildebar2}\\
&&(n+1)\cdot\min_A\Lambda_A-\bar S>0,  ~\forall ~ \mathfrak F_A\cap  \mathfrak a_+^*\neq \emptyset,  \label{barS}
\end{eqnarray}
where $\Lambda_A$ are constants defined by (\ref{constant-a}).  Then the K-energy is proper on $\mathcal H_{K\times K}(\frac{1}{2}d \eta)$ modulo $Z(H)$.
\end{theo}

We will give a proof of Theorem \ref{proper general class} in case of $\omega_g^T\in{\frac{\pi}{n+1}} c_1^B(M)$ in next section. In this case, $\Lambda_A=2(n+1)$ are all and $\bar S=2n(n+1)$. Thus (\ref{tildebar2}) and (\ref{barS}) are both automatically satisfied. Since $P$ does not satisfy the Delzant condition in general \cite{Gui}, we need to modify the argument in the proof of \cite[Theorem 1.2]{LZZ}. For a general transverse K\"ahler class $ [\omega_g^T]_B$, we left the proof to the reader.

\begin{rem}\label{mu-depends} Since the polytope $P$ in Theorem \ref{proper general class} depends on the choice of $\gamma$, we do not know whether the condtions (\ref {tildebar1})-(\ref{barS}) depend on $\gamma$ or not. But in case of $\omega_g^T\in{\frac{\pi}{n+1}} c_1^B(M)$, the conditions are independent of $\gamma$ (cf. Section 5).
\end{rem}

\section{Properness of $\mu(\cdot)$}

In this section, we prove Theorem \ref{proper general class} in case of $\omega_g^T\in{\frac{\pi}{n+1}} c_1^B(M)$. First
we use Lemma \ref{Hessian} to give a criterion to verfiy $\omega_g^T\in{\frac{\pi}{n+1}} c_1^B(M)$ in terms of the moment cone $\mathfrak C$ given by (\ref{cone def}). This criterion is in fact similar to that a  $G$-manifold being Fano is determined by its moment polytope \cite{LZZ}. We need to introduce some notations below.

Set $\mathfrak C_+=\mathfrak C\cap\mathfrak a^*_+$. We call a facet $\mathfrak F_A$ satisfies $\mathfrak F_A\cap\mathfrak a^*_+\not=\emptyset$ an \emph{outer facet} of $\mathfrak C_+$. Note that for any Weyl chamber ${\mathfrak a^*_+}'$, there exists a unique $w'\in W$ such that $w'(\mathfrak a^*_+)={\mathfrak a^*_+}'$. Thus for any $\mathfrak F_{A'}$ which intersects ${\mathfrak a^*_+}'$, $w'^{-1}(\mathfrak F_{A'})$ is an outer facet by $W$-invariance of $\mathfrak C$, and it has prime normal vector $w'^{-1}(u_{A'})$. We associate to $\mathfrak F_{A'}$ a vector $\sigma_{A'}:=w'(\sigma)$. Obviously
\begin{eqnarray}\begin{aligned}\label{0303++}
\sigma(w^{-1}u_{A'})=\sigma_{A'}(u_{A'}).
\end{aligned}\end{eqnarray}

\begin{prop}\label{Fano condition}
\begin{eqnarray}\label{transverse fano}\omega_g^T\in{\frac{\pi}{n+1}} c_1^B(M)\end{eqnarray}
holds if and only if there is a $\gamma_0\in  \mathfrak a_z^*$ such that
\begin{eqnarray}\begin{aligned}\label{0303}
\gamma_0(u_A)=-1+2\sigma_A(u_A),~\forall~ A,
\end{aligned}\end{eqnarray}
and
\begin{eqnarray}\begin{aligned}\label{0304}
\gamma_0(\xi) =-(n+1).
\end{aligned}\end{eqnarray}
\end{prop}

\begin{proof}
Suppose that (\ref{transverse fano}) is true. Then by the relations (\ref{0206}), \cite[(10)]{Futaki-Ono-Wang} and
$${\rm Ric}(\bar g)(X,Y)={\rm Ric}(g)(X,Y)-2ng(X,Y),~\forall X,Y\in TM,$$
we have
$${\rm Ric}(\bar g)=\sqrt{-1}\partial\bar{\partial}h.\footnote{For any basic function $h$ on $M$, we can extend it to $C(M)$ by assumming $\xi(h)={\frac{\partial}{\partial \rho}}h=0$.}$$
Using the K\"ahler potential $F={\frac12}\rho^2$, we get
\begin{equation*}
\partial\bar\partial (-\log {\rm det}(\partial\bar\partial F)- h)=0, ~{\rm in} ~Z,
\end{equation*}
where the operators $\partial, \bar\partial$ are both defined in the affine coordinates on $Z$.
On the other hand, by Proposition \ref{sym potential}, the growth behavior of $F$ on the torus cone $Z$ is same as $\hat F$ in (\ref{0003.3+}). Then one can check that $\log {\rm det}(\partial\bar\partial F)$ has at most the linear growth.
Thus there is $\gamma_0 \in \mathfrak a^*$ such that
$$-\log {\rm det}(\partial\bar\partial F)= h +2\gamma_0(x).$$
By (\ref{MA}), it follows that
\begin{equation}\label{0301}
-\log \mathrm{MA}_{\mathbb R}(F)-\log\prod_{\alpha\in R_G^+}\langle\alpha,\nabla F\rangle^2-\chi(x)=2
\gamma_0 (x) +h+C,~\forall x\in\mathfrak a_+,
\end{equation}
where $\chi(x)=-\log\mathbf J(x)$.
Note that the function $\gamma_{0}(x)$ is $W$-invariant. It follows that $\gamma_0\in   \mathfrak a_z^*\subset \mathfrak z^*(\mathfrak g)$.

Taking the Legendre transformation of $F$ in (\ref{0301}), we have
\begin{eqnarray}\label{0302}
\log\det(U_{,ij})-2\sum_{\alpha\in R_G^+}\log\langle\alpha,y\rangle-\chi(\nabla U)=2\gamma_{0i}U_{,i}+h+C,
\end{eqnarray}
where $U$ is the Legendre function of $F$ and $$U_{,i}=\frac{\partial U}{\partial y_i},~U_{,ij}=\frac{\partial^2 U}{\partial y_i\partial y_j}.$$
Since $\gamma_0$ is $W$-invariant, it suffices to prove (\ref{0303}) when $\mathfrak F_A$ is an outer facet. Let $y_0$ a point on  a facet  $\mathfrak F_A$  of $\mathfrak C$ in $\mathfrak a_+^*$,  which is away from other facets and all Weyl walls. Then by (\ref{0211}), it is easy to see that the sum of  singular terms at the left-hand side of (\ref{0302}) goes to
$$-\log l_A(y)+2 \sigma (u_A) \log l_A(y) $$
as as $y\to y_0$. Similar to the right-hand side of (\ref{0302}), we have
$$ \gamma_0(u_A) \log l_A(y), ~{\rm as}~ y\to y_0.$$
Thus combining the above two relations, we derive (\ref{0303}). Furthermore, one can verify that $\gamma_0$ is uniquely determined by (\ref{0303}) and (C2)-condition  for the
good cone $\mathfrak C$  in Section 2.  By (C1)-condition,  $\gamma_0$ is also rational.

Next we determine the quantity $  \gamma_0(\xi)$. We note that $ \alpha(\xi) =0$ for any $\alpha\in R_G$, since $\xi\in\mathfrak z(\mathfrak k)$. It follows
$$\xi^i{\frac{\partial\mathbf J}{\partial x^i}}=2\mathbf J(x)\sum_{\alpha\in R_G^+}\alpha(\xi)\cdot\coth\alpha(x)
=0.$$
Thus combining with (\ref{0303+}), we get
\begin{align}\label{det-u-homogeous}
&y_i{\frac{\partial}{\partial y_i}}\det(U_{,ij})\notag\\
&=y_i\frac{\partial}{\partial y_i}\left(e^{2\gamma_{0i}U_{,i}+h+C}{\frac{\prod_{\alpha\in R_G^+}\langle\alpha,y\rangle^2}{\mathbf J(\nabla U)}}\right)\notag\\
&=\left(\gamma_{0i}\xi^i+n-r-{\frac{1}{\mathbf J(\nabla U)}}{\frac{\partial\mathbf J}{\partial x^i}}{\frac{\partial x^i}{\partial y_k}}y_k\right)\left(e^{2\gamma_{0i}U_{,i}+h+C}{\frac{\prod_{\alpha\in R_G^+}\langle\alpha,y\rangle^2}{\mathbf J(\nabla U)}}\right)\notag\\
&=\left(\gamma_{0i}\xi^i+n-r\right)\det(U_{,ij}).
\end{align}
On the other hand, $\det(U_{,ij})$ is homogenous of degree $-(r+1)$. Hence by the Euler's equation,
we obtain (\ref{0304}) from (\ref{det-u-homogeous}) immediately.

To prove the sufficient part of proposition, it suffices to show that $-lK_{C(M)}$ is trivial for some $l\in\mathbb N_+$ as in the proof of \cite[Theorem 1.2]{Cho-Futaki-Ono}. We reduce the problem to show that $-lK_{C(M)}|_Z$ is trivial for some $l$. Then we extend the property to $C(M)$ by the $K\times K$-invariance through constructing a non-trivial meromorphic function on $Z$.

By the work of Brion \cite{Brion} (see also \cite[Sect. 1.8]{Ruzzi}), we have
\begin{eqnarray*}
-K_{C(M)}|_Z=\sum_A(1-2\sigma_A(u_A))D_A,
\end{eqnarray*}
where $D_A$ is the boundary prime divisor of $Z$ associated to $\mathfrak F_A$. By $(\ref{0303})$, it follows
\begin{eqnarray*}
-K_{C(M)}|_Z=-\sum_A\gamma_0(u_A) D_A.
\end{eqnarray*}
Recall that $\gamma_0$ is rational. This means that there is an $l\in\mathbb N_+$ such that $l\gamma_0$ is a lattice point in $\mathfrak N^*$. Thus there is a global meromorphic function which defines the divisor $-\sum_Al\gamma_0(u_A) D_A$ (cf. \cite[Chapter 3]{Fulton}). Hence $-lK_{C(M)}|_Z$ is trivial, and so $-lK_{C(M)}$ is. The proof is completed.
\end{proof}

By (\ref{0205}),
it is easy to see $\bar S=2n(n+1)$. To simplify the reduced energy $\mu(\cdot)$ in Proposition \ref{mu(u)}, which depends on the choice of $H$-orbit, we take a translation
$$v'=v+{\frac1{n+1}}\iota^*(\gamma_0),$$
where $\gamma_0$ is given by Proposition \ref{Fano condition}. Then we get a translated polytope $P'=P+{\frac1{n+1}}\iota^*(\gamma_0)$ from (\ref{4.4+}), which is defined by
\begin{eqnarray*}
l'_A(v')=\left(u_A^i-{\frac{ \gamma(u_A)}{\gamma(\xi)}}\xi^i\right)v'_i+{\frac{1-2\sigma_A(u_A)}{n+1}}>0,~\forall A.
\end{eqnarray*}
It is also easy to see that the pull back of any function $u\in\mathcal C_{P,W}$ lies in $\mathcal C_{P',W}$.

 The advantage of choice of $P'$ is that $\Lambda_A=2(n+1)$ for all $A$. Then
\begin{eqnarray}\begin{aligned}
\mathcal L(u)={\frac{2(n+1)}{V_P}}\int_{P'_+}\left\langle v'-\frac2{n+1}\sigma,\nabla u\right\rangle\pi\,dv',\notag
\end{aligned}\end{eqnarray}
and
\begin{align}\label{k-energy-c1}
&\mu(u)\notag\\
&=\mathcal L(u)-{\frac{1}{V_P}}\int_{P'_+}\log\det\left(u_{ij}\right)\pi\,dv'+{\frac{1}{V_P}}\int_{P'_+}[\chi\left(\nabla u \right)+4\sigma(\nabla u) ]\pi\,dv'.
\end{align}
One can check that (\ref{k-energy-c1}) is just the reduced K-energy associated to the $H_0$-orbit determined by choosing $\gamma=\gamma_0$. In the later, we always assume that $H=H_0$.

By the fact that $\bar S=2n(n+1)$ and $\Lambda_A=2(n+1)$, we see that (\ref{tildebar1}) is equivalent to
\begin{eqnarray}\label{0316+}
bar(P_+)=\frac{\int_{ P_+}v \pi\,dv}{\int_{ P_+}\pi\,dv}\in{\frac2{n+1}}\sigma+\Xi.
\end{eqnarray}
Moreover, (\ref{0507}) is equivalent to
\begin{eqnarray}\label{5.8+}
bar(P_+)\in\mathfrak a_{ss,+},
\end{eqnarray}
where $\mathfrak a_{ss,+}$ is the semi-simple part of $\mathfrak a'_{+}$.
Hence, (\ref{0316+}) implies that $M$ has vanishing Futaki invariant by Lemma \ref{Fut=0 and L=0}.
By Lemma \ref{normalization of u}, Theorem \ref{proper general class} turns to prove the following proposition in case of (\ref{transverse fano}).

\begin{prop}\label{main-proposition} Assume that (\ref{0316+}) is satisfied. Then
$\mu(\cdot)$ is proper on $\hat{\mathcal C}_{P,W}$. More precisely, there are $\delta, C_\delta>0$ such that
$$\mu(u)\ge \delta \int_{P_+} u \pi(v)\,dv -C_\delta, ~\forall~ u\in \hat{\mathcal C}_{P,W}.$$
\end{prop}

\subsection{A criterion for the properness of general functionals}

In this subsection, we will establish a criterion to verify the properness of general functionals $\mu(\cdot)$ for convex functions on a bounded polytope $P$. Let us introduce a setting for such a $P$ and related functionals as follows.

Let $H=(K')^c$ be a reductive Lie group of dimension $n$ with $T'$ its maximal compact torus, and assume that the rank of $H$ is $r$.
Let $R\subset J(\mathfrak t')^*=(\mathfrak a')^*$ be the root system and $R^+$ a chosen set of positive roots. Set $2\sigma=\sum_{\alpha\in R^+}\alpha$ and denote the corresponding Weyl group by $W$. We assume that a bounded polytope $P\subset(\mathfrak a')^*$, which can be described as
\begin{eqnarray*}
P=\bigcap_{\{A=1,...,d\}}\{l'_A(v):=\lambda_A-u_A^iv_i>0\}
\end{eqnarray*}
with each $\lambda_A>0$, which satisfies:
\begin{itemize}
\item[(P1)] $P$ is convex and $W$-invariant, which contains the origin $O$;
\item[(P2)] Each codimension $N$ face of $P$ is exactly intersections of $N$ facets. In particular, each vertex of $P$ is exactly the intersection of $r$ facets;
\item[(P3)] Each $u_A$ satisfies
\begin{eqnarray}\label{0802}
\alpha(u_A)\in\mathbb Z,~\forall ~\alpha\in R.
\end{eqnarray}
\end{itemize}
We note that $u_A$ need not to be a lattice vector in the lattice of one parameter groups. Also we remark that the moment polytope $P$ given in Section \ref{4.2} satisfies these conditions. As before, we set $P_+=P\cap(\mathfrak a')^*_+$, where $(\mathfrak a')^*_+$ is the positive Weyl chamber defined by $R^+$.

Define the Guillemin function of $P$ by
\begin{eqnarray*}
u_P(v)={\frac12}\sum_Al_A(v)\log l_A(v).
\end{eqnarray*}
Then it has properties:
\begin{itemize}
\item[(F1)] $u_P\in C^\infty(P)\cap C^0(\overline P)$;
\item[(F2)] $u_P$ is $W$-invariant and strictly convex;
\item[(F3)] The derivatives of $u_P$ satisfies
\begin{eqnarray*}
u_P^{ij}\in C^\infty(\overline P),
\end{eqnarray*}
where $u_{P,ij}=\frac{\partial^2}{\partial v_1\partial v_j}u_{P}$ and $(u_P^{ij})=(u_{P,ij})^{-1}$.
\end{itemize}

Let $a=(a^i)\in\mathfrak a_z'=\mathfrak a'\cap\mathfrak z(\mathfrak h)$, the central part of $\mathfrak a'$. Assume that
$$a_1\leq a^iv_i\leq a_2,~\forall ~v\in\overline P$$
for some $a_1,a_2$. Let $f(t):[a_1,a_2]\to\mathbb R$ be a function which satisfies:
\begin{itemize}
\item [(W1)] $f$ is smooth;
\item [(W2)] There are constants $m_f,M_f$ such that
\begin{eqnarray*}
0<m_f\leq f(t)\leq M_f,~\forall t\in [a_1,a_2];
\end{eqnarray*}
\item [(W3)] There is constants $C_f$ such that
\begin{eqnarray*}
||f(t)||_{C^2}\leq C_f.
\end{eqnarray*}
\end{itemize}
For simplicity, we denote $f_a(v)=f(a^iv_i)$.

Set a space of normalized $W$-invariant strictly convex functions by
\begin{eqnarray}\begin{aligned}
\hat {\mathcal C}'_{P,W}=\{u\in C^{\infty}(P)\cap C^{0}(\overline P)|~
&u\text{ is strictly convex and $W$-invariant on}~ P ,\notag\\
&u \geq u(O)=0\}.\notag
\end{aligned}\end{eqnarray}
Let $\pi, \chi$ be functions as before. Given $f_a$ and a constant $\Lambda_L>0$, we define a weighted functional $\mu(\cdot)$  associated to $f_a$ for any $u\in\hat{\mathcal C}'_{P,W}$ by
\begin{eqnarray*}
\mu (u)=\frac1{\Lambda_L}\mathcal L(u)+\mathcal N(u),
\end{eqnarray*}
where
\begin{eqnarray}\label{L(u)}
\mathcal L(u)=\int_{P_+}\left\langle v-{{4}{\Lambda_L}}\sigma,\nabla u\right\rangle f_a(v)\pi(v)dv,
\end{eqnarray}
and
\begin{eqnarray*}\begin{aligned}
\mathcal N(u)
&=-\int_{P_+}\log\det\left(u_{,ij}\right)f_a(v)\pi(v)dv\notag\\&+\int_{P_+}[\chi\left(\nabla u\right)+4\sigma(\nabla u) ]f_a(v)\pi(v)dv.
\end{aligned}\end{eqnarray*}
Clearly, $\mathcal L(\cdot)$ is well-defined on $\hat{\mathcal C}'_{P,W}$. We will show that $\mathcal N(\cdot)$ is also well-defined, so is $\mu(\cdot)$ below. The following is the main result in this section.

\begin{theo}\label{proper principle} Let $\Xi$ is the relative interior of the cone generated by $R^+$.
Suppose that $P_+$ satisfies
\begin{eqnarray}\label{weighted bar condition}
bar_a(P_+)=\frac{\int_{P_+}vf_a(v)\pi(v)dv}{\int_{P_+}f_a(v)\pi(v)dv}\in{4{\Lambda_L}}\sigma+\Xi.
\end{eqnarray}
Then there is a $\delta>0$ and a constant $C_\delta$ such that
\begin{eqnarray}\label{0803}
\mu(u)\geq\delta\int_{P_+}u f_a(v)\pi(v)dv-C_\delta,~\forall u\in\hat{\mathcal C}'_{P,W} .\notag
\end{eqnarray}
\end{theo}
Clearly, Proposition \ref{main-proposition} follows from Theorem \ref{proper principle} by taking $f\equiv1$, $\Lambda_L=(2(n+1))^{-1}$ and $P=\iota^*(\mathcal P)$. In the following, we will use the arguments in \cite{LZZ} to prove the theorem.

\subsection{The linear part $\mathcal L(\cdot)$.}

Let $d\sigma_0$ be the Lebesgue measure of $\partial P_+$ and $\nu$ the corresponding unit normal vector. By (W2)-condition for $f_a$ and convexity of $u$, there is a constant $\Lambda$ such that for any $W$-invariant convex function $u$ which is normalized at $O$,
\begin{eqnarray}\label{0502+}
\int_{P_+} uf_a(v)\pi(v)dv\leq\Lambda\int_{\partial P_+} u \langle v,\nu\rangle f_a(v)\pi(v)d\sigma_0.
\end{eqnarray}

Taking integration by parts in (\ref{L(u)}), and  using the fact that
\begin{eqnarray*}
v_i\pi_{,i}(v)=(n-r)\pi(v),
\end{eqnarray*}
we have
\begin{eqnarray}\begin{aligned}
&4\Lambda_L\int_{P_+}u\,\sigma(\nabla\pi)\rangle f_a(v)dv\notag\\
&=\mathcal L(u)-\int_{\partial P_+}u\langle v-{4\Lambda_L}\sigma,\nu\rangle f_a\pi d\sigma_0\notag\\
&+n\int_{P_+}u f_a(v)\pi(v)dv+\int_{P_+}u\langle v-4\Lambda_L\sigma,a\rangle f'\pi dv.\notag
\end{aligned}\end{eqnarray}
Then by using (W2), (W3) and (\ref{0502+}), we get a constant $C>0$ such that
\begin{align}\label{dpi estimate}
&\int_{P_+} u\,\sigma(\nabla\pi) f_a(v)dv\notag\\
&\leq {\frac{1}{4\Lambda_L}}\mathcal L(u)+C\int_{\partial P_+} u\langle v,\nu\rangle f_a(v)\pi(v)d\sigma_0,~\forall u \in\hat{\mathcal C}'_{P,W}.
\end{align}
Combining (\ref{0502+}), (\ref{dpi estimate}) and following the argument in the proof of \cite[Proposition 4.3]{LZZ}, we can prove

\begin{lem}\label{linear prop}
Under the assumption (\ref{weighted bar condition}), there is a constant $\lambda>0$, such that
\begin{eqnarray*}
\mathcal L(u)\geq\lambda\int_{\partial P_+}u\langle v,\nu\rangle f_a(v)\pi(v)d\sigma_0,~\forall u\in\hat{\mathcal C}'_{P,W}.
\end{eqnarray*}
\end{lem}

\subsection{The nonlinear part $\mathcal N(\cdot)$.}

In this subsection, we estimate $\mathcal N(\cdot)$. In particular, we show that $\mathcal N(\cdot)$ is well-defined on $\hat{\mathcal C}'_{P,W}$. We will use a method in \cite{D1} (also see \cite{ZZ2, LZZ}). In fact, it suffices to show that for any $u\in\hat{\mathcal C}'_{P,W}$,
\begin{align}\label{0812}
\mathcal N^+(u)&=-\int_{P_+}[\log\det\left(u_{,ij}\right)-\chi\left(\nabla u\right)-4\sigma(\nabla u) ]^+f_a(v)\pi(v)dv\notag\\
&>-\infty.
\end{align}

As in the proof of \cite[Lemma 6.3]{LZZ}, for any $u\in\hat{\mathcal C}'_{P,W}$, we define a $W$-invariant function $\hat u$ such that
$$\hat u|_{P_+}=u+{\frac12}c|v|^2+\sigma_iv_i.$$
Then $\hat u$ lies in $C^\infty(P_+)\cap C^0(\overline P_+)$ and satisfies
$$\mathcal N^+(\hat u)<\mathcal N^+(u).$$
Thus by replacing $u$ with $\hat u$, we may assume $$\log\det\left(u_{,ij}\right)-\chi\left(\nabla u\right)-4\sigma(\nabla u) >0,$$
and consequently $\mathcal N^+(u)=\mathcal N(u)$.

By the convexity of $\chi(\cdot)$ and
$-\log\det(\cdot)$, we have
\begin{align}\label{+42}
&-\log\det(u_{,ij})+\chi(\nabla u)\notag\\
&\geq-\log\det(u_{0,ij})+\chi(\nabla u_0)-\phi_0^{ij}(u_{,ij}-u_{0,ij})+\left.{\frac{\partial\chi}{\partial x^i}}\right|_{x=\nabla u_0}(u_{,i}-{u_{0,i}}).
\end{align}
On the other hand, by the condition (P2), we have (cf. \cite{D2}),
\begin{eqnarray}\label{boundary behavior of uij}
u_P^{ij}\nu_{Ai} \to 0~\text{ and }~u^{ij}_{P,j}\nu_{Ai} \to -\frac{2}{\lambda_A}\langle v, \nu_{ A}\rangle,
\end{eqnarray}
as $v$ goes to a facet $\mathfrak F'_A=\{v|~l'_A(v)=0\}$ of $P$. Here $\nu_A$ is the unit outer normal vector of $\mathfrak F'_A$.
Thus integrating both sides of (\ref{+42}) on $P_+$ and taking integration by parts for the terms $u_0^{ij}u_{,ij}$ and $ \frac{\partial\chi}{\partial x^i} |_{x=\nabla u_0}u_{,i}$, we get
\begin{align}\label{0804}
&\int_{P_+}[-\log\det(u_{,ij})+\chi(\nabla u)]f_a\pi dv\notag\\
&\ge -\int_{\partial P_+}u_{0}^{ij} u_{,i}\nu^jf_a\pi d\sigma_0+\int_{P_+}u_{0,j}^{ij} u_{,i}f_a\pi dv
+\int_{P_+}u_{0}^{ij}a^{j}u_{,i}f'\pi dv\notag\\
&+\int_{P_+}u_{0}^{ij} u^{,j}\pi_{,i}f_a dv
+\int_{\partial P_+}u\left.{\frac{\partial\chi}{\partial x^i}}\right|_{x=\nabla u_{0}}\nu^if_a\pi d\sigma_0\notag\\
&-\int_{ P_+} u\left.{\frac{\partial^2\chi}{\partial x^i\partial x^j}}\right|_{x=\nabla u_{0}} f_a\pi dv\notag\\
&-\int_{P_+} u\left.{\frac{\partial\chi}{\partial x^i}}\right|_{x=\nabla u_{0}}a^if'\pi dv-\int_{P_+} u\left.{\frac{\partial\chi}{\partial x^i}}\right|_{x=\nabla u_{0}}f_a\pi_{,i} dv-C_0.
\end{align}
We need to deal with each terms in (\ref{0804}) in the following.

Note that $u$ is convex and continuous on $\overline P$. Then by  (\ref{boundary behavior of uij}), we have (cf. \cite[Lemma 3.3.5]{D1}),
\begin{eqnarray}\label{0805}
-\int_{\partial P_+} u_{0}^{ij} u_{,i}\nu^jf_a\pi d\sigma_0&=&0
\end{eqnarray}
and
\begin{eqnarray}\label{0805=1}
 -\int_{\partial P_+} u_{0,i}^{ij} u\nu^jf_a\pi d\sigma_0&=&\sum_A{\frac2{\lambda_A}}\int_{\mathfrak F'_A\cap(\mathfrak a')^*_+}u f_a\pi\,d\sigma_0.
\end{eqnarray}
Note that $\pi$ vanishes quadratically on Weyl walls and
\begin{eqnarray}\label{0806}
{\frac{\partial\chi}{\partial x^i}}(x)\to-4\sigma_i
\end{eqnarray}
as $x\to\infty$ and away from Weyl walls. We see that
\begin{eqnarray}\label{0807}
\int_{\partial P_+} u\left.{\frac{\partial\chi}{\partial x^i}}\right|_{x=\nabla u_{0}}\nu^if_a\pi d\sigma_0\to-\int_{\partial P_+}4 u\sigma_i\nu^if_a\pi d\sigma_0.
\end{eqnarray}
Moreover, by the fact that $\alpha(a)=0$ for any $\alpha$, we have
\begin{eqnarray}\label{0808}
\frac{\partial\chi}{\partial x^i}a^i=-2\sum_{\alpha\in R^+} (\alpha_ia^i)\coth\alpha(x)=0.
\end{eqnarray}

On the other hand, taking integration by parts with help of \eqref{boundary behavior of uij}, and then by (F3), (W3),
we get the following estimates,
\begin{align}\label{0809}
&\left|\int_{P_+}u_{0,j}^{ij}u_{,i}f_a\pi dv \right|\notag\\
&=\left|\sum_A{\frac2{\lambda_A}}\int_{\mathfrak F'_A\cap(\mathfrak a')^*_+}u f_a\pi\,d\sigma_0\right.\notag\\
&-\left.\int_{P_+}[u_{0,ij}^{ij}f_a\pi+u_{0,j}^{ij}a^if'\pi+u_{0,j}^{ij}f_a\pi_{,j}]dv\right|\notag\\
&\le C\left(\int_{ \partial P_+} u \langle v,\nu \rangle f_a(v)\pi(v) d\sigma_0+
\int_{P_+} u(1+\sigma(\nabla \pi) ) f_a(v)dv\right),
\end{align}
\begin{align}\label{0809-2}
&\int_{P_+}u_{0}^{ij}a^{j}u_{,i}f'\pi dv\notag\\
&=-\int_{P_+}u[u_{0,i}^{ij}a^{j}f'\pi+u_{0}^{ij}a^{j}f'\pi_{,i}+u_{0}^{ij}a^ia^{j}f''\pi ]dv\notag\\
&\le C\left(\int_{ P_+}u\langle v,\nu\rangle f_a(v)\pi(v) dv+
+ \int_{P_+} u\,\sigma(\nabla \pi) f_a(v)dv\right),
\end{align}
\begin{align}\label{0809-3}
&\int_{P_+}u_{0}^{ij}u_{,i} f_a\pi_{,j} dv\notag\\
&=-\int_{P_+}u_{0,i}^{ij}u f_a\pi_{,j} dv-\int_{P_+} u[u_{0}^{ij} f'a^j\pi_{,j}+u_{0}^{ij}f_a\pi_{,ij}] dv.\notag
\\
&\le C\left(\int_{ P_+} u\langle v,\nu\rangle f_a(v)\pi(v) dv+
+ \int_{P_+} u\,\sigma(\nabla \pi) f_a(v)dv\right).
\end{align}
Thus substituting (\ref{0805})-(\ref{0809-3}) into (\ref{0804}), we finally obtain
\begin{eqnarray}\begin{aligned}
\mathcal N(u) &\geq-C_{1}\int_{\partial P_+} u\langle v,\nu\rangle f_a(v)\pi(v) d\sigma_0\notag\\
&-C_2 \int_{P_+} u(1+\sigma(\nabla \pi)) f_a(v)dv +\int_{P_+} uQ f_a(v)\pi(v)dv -C_3,\notag
\end{aligned}\end{eqnarray}
where
\begin{eqnarray}\label{59}
Q=-\left.{\frac{\partial \chi}{\partial x^i}}\right|_{x=\nabla u_0}{\frac{\pi_{,i}}{\pi}}-\left.{\frac{\partial^2 \chi}{\partial x^i\partial x^k}}\right|_{x=\nabla u_0}{\frac{\partial^2u_0}{\partial v_i\partial v_k}}- u_0^{ij}{\frac{\pi_{,ij}}{\pi}}.
\end{eqnarray}
Hence by (\ref{0502+}) and (\ref{dpi estimate}), we see that there are uniform constants $C_1,C_2,C_3>0$ such that for any $u\in\hat{\mathcal C}'_{P,W}$,
\begin{align}\label{0503+}
&\mathcal N^+(u)\notag\\
&\geq-C_{1}\int_{\partial P_+} u\langle v,\nu\rangle f_a(v)\pi(v) d\sigma_0-C_2\mathcal L(u) + \int_{P_+} uQ f_a(v)\pi(v)dv +C_3.
\end{align}
In particular, (\ref{0812}) is true.

\subsubsection{Estimate of $Q$.}
As in \cite{LZZ}, we have to control the growth of $Q$ near Weyl walls. The goal is to show that

\begin{lem}\label{0504+}
There is a uniform constant $C_Q$ such that
\begin{eqnarray*}
|Q|\pi\leq C_Q\,\sigma(\nabla \pi),~\forall v\in P_+.
\end{eqnarray*}
\end{lem}

\begin{proof}
From (\ref{59}), a direct computation shows
\begin{align*}
Q&=\sum_{\alpha\in R^+}\left[4{\frac{|\alpha|^2\coth \alpha(\nabla u_0)}{\langle\alpha,v\rangle}}-2{\frac{u_{0,ij}\alpha_i\alpha_j}{\sinh^{2}\alpha(\nabla u_0)}}-2u_0^{ij}{\frac{\alpha_i\alpha_j}{\langle\alpha,v\rangle^2}}\right]\nonumber\\
&+2\sum_{\alpha\neq \beta\in R^+}\left[\coth{\alpha(\nabla u_0)}\cdot{\frac{\langle \alpha,\beta\rangle}{\langle\beta,v\rangle}}+{\coth{\beta(\nabla u_0)}}\cdot{\frac{\langle \alpha,\beta\rangle}{\langle\alpha,v\rangle}}-2u_0^{ij}{\frac{\alpha_i\beta_j}{\langle\alpha,v\rangle\langle\beta,v\rangle}}\right].
\end{align*}
For simplicity, we denote each term in these two sums by $I_\alpha(v)$ and $I_{\alpha,\beta}(v)$, respectively.

To estimate $I_\alpha(v)$, it suffices to control it near the Weyl wall $W_\alpha=\{v|~\langle\alpha,v\rangle=0\}$. By the $W$-invariance of $P$, we can divide outer faces of $P$ exactly into three classes as in \cite{LZZ}. Fix a point $v_0\in W_\alpha$, let $v\to v_0$. Following the arguments of \cite[Lemma 4.9, Lemma 4.11]{LZZ}, we see that
there is a neighbourhood $U_{v_0}$ and a constant $C_{v_0}$ such that
\begin{eqnarray*}
|I_\alpha(v)|\leq {\frac{C_{v_0}}{\langle\alpha,v\rangle}},~\forall v\in U_{v_0}\cap P_+.
\end{eqnarray*}
We should remark that  by our assumption (\ref{0802}) it holds
\begin{eqnarray*}
\alpha(u_A)\in \mathbb Z_{>0},
\end{eqnarray*}
 for any  outer facet  $\mathfrak F_A$ which is  not orthogonal to $W_\alpha$, although $u_A$ may  not be a lattice vector. Thus the arguments in Case (iii) of \cite[Lemma 4.11]{LZZ} are still available. Following \cite[Lemma 4.11]{LZZ}, $I_{\alpha,\beta}$ can be estimated in a similar way. Since $\partial P_+\cap W_\alpha$ is compact, there are uniform constants $C_\alpha,C_{\alpha,\beta}$ such that
\begin{eqnarray}\label{I-alpha}
\begin{aligned}&|I_\alpha(v)|\leq {\frac{C_{\alpha}}{\langle\alpha,v\rangle}},\\
&|I_{\alpha,\beta}(v)|\leq C_{\alpha,\beta}\left({\frac{1}{\langle\alpha,v\rangle}}+{\frac{1}{\langle\beta,v\rangle}}\right),
\end{aligned}
\end{eqnarray}
for any $v\in P_+$.

Recall that $\langle\sigma,\alpha\rangle>0$ for any $\alpha\in R^+$. Then
\begin{eqnarray*}
{\frac{\sigma(\nabla\pi(v))}{2\pi(v)}}=\sum_{\alpha\in R^+}{\frac{\langle\sigma,\alpha\rangle}{\langle\alpha,v\rangle}}\geq {\frac C {\langle\alpha,v\rangle}},~\forall\alpha\in R^+.
\end{eqnarray*}
Thus Lemma \ref{0504+} follows from (\ref{I-alpha}) and the above inequality.

\end{proof}

Combining (\ref{0502+}), (\ref{dpi estimate}), (\ref{0503+}) and Lemma \ref{0504+}, we prove

\begin{prop}\label{0505+}
There are uniform constants $C_0,C_L>0$ such that for any $u\in\hat{\mathcal C}'_{P,W}$,
\begin{eqnarray}\label{n-l-relation}
\mathcal N^+(u)\geq
-C_L\mathcal L(u)-C_0.
\end{eqnarray}
\end{prop}

(\ref{n-l-relation}) implies (\ref{0812}). Thus $\mathcal N(\cdot)$ is well-defined on $\hat{\mathcal C}'_{P,W}$.

\subsection{Proof of Theorem \ref{proper principle}}

\begin{proof}[Proof of Theorem \ref{proper principle}]
Let $\epsilon\in(0,1)$ be a small positive number. Note that $\mathcal N(\epsilon u)>\mathcal N^+(\epsilon u)$. Then by Proposition \ref{0505+}, it is not hard to see that (cf. \cite[Proposition 4.1]{LZZ}),
\begin{eqnarray*}
\mathcal N(u)&\geq& \mathcal N(\epsilon u)+n\log\epsilon\notag\\
&\geq&-C_0+n\log\epsilon-\epsilon C_L\mathcal L(u).
\end{eqnarray*}
Take $\epsilon$ sufficiently small such that
$$1-\epsilon\cdot\Lambda_LC_L=\delta'\cdot\Lambda_L>0.$$
Thus we get
\begin{eqnarray*}
\mu(u)&=&{\frac1{\Lambda_L}}\mathcal L(u)+\mathcal N(u)\notag\\
&\geq&\delta'\mathcal L(u)-C_0+n\log\epsilon.
\end{eqnarray*}
Combining (\ref{0502+}) and Lemma \ref{linear prop}, we derive
\begin{eqnarray*}
\mu(u)
\geq \frac{\delta'\lambda}{\Lambda}\int_Muf_a\pi\,dv-C_0+n\log\epsilon.
\end{eqnarray*}
The theorem is proved.
\end{proof}

\section{Existence of $G$-Sasaki Einstein metrics}

To the authors' knowledge, Futaki, Ono and Wang are the first ones  in the literature who used the equation (\ref{0209}) to study the existence problem of Sasaki Einstein metrics  \cite{Futaki-Ono-Wang}. As in case of K\"ahler-Einstein metrics \cite{Yau, T1}, they
solved the following family of equations  via the continuity method,
\begin{equation}\label{0210}
\det(g_{i\bar j}^T+\psi_{,i\bar j})=\exp(-2t(n+1)\psi+h)\det(g_{i\bar j}^T),~t\in[0,1],
\end{equation}
where $g^T$ is a transverse K\"ahler metric with its K\"ahler form $ \omega_g^T \in{\frac{\pi}{n+1}} c_1^B(M)$.
It is known that $\eqref{0210}$ is solvable for sufficiently small $t>0$ and $ \omega_g^T+\sqrt{-1}\partial\bar\partial \psi $ satisfies the Sasaki Einstein metric equation (\ref{0205}) if $\psi$ is a solution of $\eqref{0210}$ at $t=1$. Thus solving (\ref{0205}) turns to do a prior-estimate for solutions $\psi_t$ for $t\in [t_0,1]$ for some $t_0>0$. As shown in  \cite{Yau, T1}, we need to do the $C^0$-estimate for solutions $\psi_t$.

As a version of Tian's theorem in case of Sasaki manifolds \cite{T1}, Zhang proved the following analytic criterion for the existence of Sasaki Einstein metrics \cite{Zhang}.

\begin{theo}\label{Zhang Xi thm}
Let $(M, \frac{1}{2}d\eta)$ be a $(2n+1)$-dimensional compact Sasaki manifold with ${\frac12}[d\eta]_B={\frac{\pi}{n+1}}c_1^B(M)$. Suppose that there is no non-trivial Hamiltonian holomorphic vector field on $M$. Then $(M, \frac{1}{2}d\eta)$ has a Sasaki-Einstein metric if and only if K-enegry $\mathcal K( \cdot)$ is proper on $\mathcal H\left(\frac{1}{2}d\eta\right)$.
\end{theo}

\subsection{A generalization of Zhang's theorem}
In general, $(M, \frac{1}{2}d\eta)$ may admit Hamiltonian holomorphic vector fields. Note that $\mathcal K( \cdot)$ is invariant under ${\rm Aut}^T(M)$ if the Futaki-invariant vanishes. Thus one shall modify  Theorem  \ref{Zhang Xi thm} for the properness property of $\mathcal K( \cdot)$ in sense of Definition \ref{properness-definition}.

Similar to $I$-functional, one can define Aubin's $J$-functional on $\mathcal H \left(\frac{1}{2}d\eta\right)$ by
\begin{eqnarray*}
J(\psi)=\int_0^1{\frac1s}I(s\psi)ds.
\end{eqnarray*}
It can be checked that (cf. \cite{Zhang})
\begin{eqnarray}
0\leq{\frac1{n+1}}I(\psi)\leq I(\psi)-J(\psi)\leq{\frac n{n+1}}I(\psi).\notag
\end{eqnarray}

\begin{lem}\label{0810}Let ${\rm Aut}_0^T(M)$ be the connected component of ${\rm Aut}^T(M)$ which contains the identity.   Then for any $\psi\in\mathcal H\left(\frac{1}{2}d\eta\right)$, there exists a $\sigma_0\in {\rm Aut}_0^T(M)$ such that
\begin{eqnarray}\begin{aligned}\label{inf-existence}
(I-J) (  \psi_{\sigma_0}) = \min_{\sigma\in {\rm Aut}_0^T(M)} \{(I-J)(  \psi_{\sigma}  ) \},
\end{aligned}\end{eqnarray}
where $\psi_{\sigma}$ is an induced potential defined by
$$\frac{1}{2}\sigma^* d\eta_\psi=\omega_g^T+\sqrt{-1}\partial\bar\partial \psi_\sigma.$$
Moreover, (\ref{inf-existence}) holds if and only if
\begin{eqnarray}\begin{aligned}\label{orthogonal-condition}
\int_M  {\rm real} (X)( \psi_{\sigma_0}) (d\eta_{\psi_{\sigma_0}})^{n}\wedge\eta_{\psi_{\sigma_0}}=0,~\forall~X\in \mathfrak{ham}(M).
\end{aligned}\end{eqnarray}

\end{lem}

\begin{proof}Let $\sigma_s$ be the one parameter subgroup in ${\rm Aut}_0^T(M)$ generated by ${\rm real}(X)$. Then
by a direct computation, we have
\begin{align}\label{0811}
&\frac{d}{ds}[I(\psi_{\sigma_s})-J(\psi_{\sigma_s})]|_{s=0}\notag\\
&=-{\frac n{2^{n-1}V}}\int_M \psi_{\sigma_0} dd^c_B\dot \psi_{\sigma_s} |_{s=0}\wedge d\eta_{\psi_{\sigma_{0}}}^{n-1}\wedge\eta_{ \phi_{\sigma_{s_0}}}\notag\\
&={\frac {\sqrt{-1}n}{2^{n-1}V}}\int_M\partial_B \psi_{\sigma_0}\wedge\bar\partial_B\dot \psi_{\sigma_s}|_{s=0}\wedge d\eta_{ \psi_{\sigma_0}}^{n-1}\wedge\eta_{ \psi_{\sigma_0}}\notag\\
&={\frac {1}{2^nV}} \int_M {\rm real}(X)(\psi_{\sigma_0}) d\eta_{\psi_{\sigma_0}}^{n}\wedge\eta_{\psi_{\sigma_0}}.
\end{align}
Thus if $\sigma_0$ is a minimizer of $F(\sigma)=I(\psi_{\sigma})-J(\psi_{\sigma})$, then (\ref{orthogonal-condition}) holds.
Conversely, we need to show that a critical point of $F(\sigma)$ is also a minimizer. This follows from the convexity of $F(\sigma)$ along any one parameter subgroup $\sigma_s$. Namely, we have
\begin{eqnarray}\begin{aligned}\label{convex-j}
{\frac{d^2}{ds^2}}[I( \psi_{\sigma_s} )-J( \psi_{\sigma_s})]\ge 0,~\forall s\ge ~0.
\end{aligned}\end{eqnarray}

Rewrite the second identity in (\ref{0811}) as
\begin{eqnarray*}
{\frac{d}{ds}}[I( \psi_{\sigma_s})-J( \psi_{\sigma_s})]=-{\frac 1{2^nV}}\int_M\dot \psi_{\sigma_s} \bigtriangleup_B \psi_{\sigma_s} d\eta_{ \psi_{\sigma_s} }^{n}\wedge\eta_{ \psi_{\sigma_s}}.
\end{eqnarray*}
Then
\begin{align}\label{second-variation-i-j}
&{\frac{d^2}{ds^2}}[I (\psi_{\sigma_s})-J( \psi_{\sigma_s})]\notag\\
&=-{\frac 1{2^nV}}\int_M\ddot \psi_{\sigma_s} \bigtriangleup_B \psi_{\sigma_s} d\eta_{ \psi_{\sigma_s}}^{n}\wedge\eta_{ \psi_{\sigma_s} }-{\frac 1{2^nV}}\int_M\dot \psi_{\sigma_s} \bigtriangleup_B\dot \psi_{\sigma_s} d\eta_{ \psi_{\sigma_s} }^{n}\wedge\eta_{ \psi_{\sigma_s} }\notag\\
&-{\frac 1{2^nV}}\int_M\dot \psi_{\sigma_s} \bigtriangleup_B\dot \psi_{\sigma_s} \bigtriangleup_B \psi_{\sigma_s} d\eta_{\psi_{\sigma_s}}^{n}\wedge\eta_{ \psi_{\sigma_s}}\notag\\
&+{\frac1{2^nV}}\int_M\dot \psi_{\sigma_s} \langle\partial_B\bar\partial_B \psi_{\sigma_s},\partial_B\bar\partial_B\dot \psi_{\sigma_s}\rangle_{d\eta_{ \psi_{\sigma_s} }} d\eta_{ \psi_{\sigma_s}}^{n}\wedge\eta_{ \psi_{\sigma_s}}.
\end{align}
Note
\begin{eqnarray*}
\ddot \psi_{\sigma_s} =|\bar\partial_B\dot \psi_{\sigma_s} |^2_{d\eta_{ \psi_{\sigma_s}}}.
\end{eqnarray*}
Taking integration by parts in (\ref{second-variation-i-j}), we get
\begin{eqnarray*}
{\frac{d^2}{ds^2}}[I(\psi_{\sigma_s} )-J( \psi_{\sigma_s})]={\frac 1{2^nV}}\int_M|X|^2_{d\eta_{ \psi_{\sigma_s}}} d\eta_{ \psi_{\sigma_s}}^{n}\wedge\eta_{ \psi_{\sigma_s}}\geq0,~\forall s\ge 0.
\end{eqnarray*}
This verifies (\ref{convex-j}).

The existence of minimizers $\sigma_0$ of $F(\sigma)$ follows from the fact
$I(\phi_{\sigma} )$ goes to the infinity when ${\rm dist}({\rm Id}, \sigma)$ goes to the infinity.

\end{proof}

The following is a modification of Theorem \ref{Zhang Xi thm} in the sufficient part.

\begin{prop}\label{proper implies existence}
Let $(M, \frac{1}{2}d\eta)$ be a $(2n+1)$-dimensional compact Sasaki manifold with ${\frac12}[d\eta ]_B={\frac{\pi}{n+1}}c_1^B(M)$. Let $K$ and $G_0$ be two subgroups of ${\rm Aut}^T(M)$ as in Definition \ref{properness-definition}. Then $(M, \frac{1}{2}d\eta)$ admits a transverse Sasaki-Einstein metric if
$\mathcal K(\cdot)$ is proper on $\mathcal H_{K}\left(\frac{1}{2}d\eta\right)$ modulo $G_0$.
\end{prop}

\begin{proof} The proof is a slight modification of Tian's argument for K\"ahler-Einstein metrics in \cite[Theorem 2.6]{Tian book} (also  see \cite{T1, Zhang}). Without loss of generality, we may assume that $d\eta$ is $K$-invariant. Thus all $\psi_t$ of (\ref{0210}) are $K$-invariant. It suffices to get a uniform bound of $I(\psi_t)$. We note that ${\rm Fut}(\cdot)\equiv 0$ on $\mathfrak {ham}(M)$ since $\mathcal K(\cdot)$ is proper on $\mathcal H_{K}\left(\frac{1}{2}d\eta\right)$ modulo $G_0$.

From the computation for solutions $\psi_t$ in (\ref{0811}), we have
$$\frac{d}{ds}[I(\psi_{\sigma_s})-J(\psi_{\sigma_s})]|_{s=0}=-{\frac 1{V}}\int_M {\rm real}(X)(\psi_t) d\eta_{\psi_t}^{n}\wedge\eta_{\psi_t}.$$
Note that
\begin{eqnarray*}
h_t+2(n+1)(1-t)\psi_t=c_t,
\end{eqnarray*}
where $h_t$ is the basic Ricci potential of $\frac{1}{2}d\eta_{\psi_t}$ and $c_t$ is a constant.
Thus
\begin{eqnarray}\begin{aligned}\frac{d}{ds}[I(\psi_{\sigma_s})-J(\psi_{\sigma_s})]|_{s=0}&={\frac1{(1-t)V}} {\rm real}( {\rm Fut}(X))\notag\\
&=0, \forall~X\in \mathfrak{ham}(M).\notag
\end{aligned}\end{eqnarray}
This means that $\psi_t$ is a minimizer of $I(\psi_{\sigma})-J(\psi_{\sigma})$  for $\psi_t$ by Lemma \ref{0810}.
Since $\mathcal K(\psi_t)$ is uniformly bounded above for any $t\in [t_0,1]$ (cf. \cite{Zhang}), $I(\psi_{t})-J(\psi_{t})$ and so $I(\psi_t)$ is uniformly bounded by the properness of $\mathcal H_{K}\left(\frac{1}{2}d\eta\right)$ modulo $G_0$.

\end{proof}

\subsection {Proof of Theorem \ref{main thm}} First we prove the necessary part. Here we will use an argument for extremal K\"ahler metrics from \cite{ZZ, ZZ1}. In fact, we have the following proposition.

\begin{prop}\label{ness}
Suppose that $M$ admits a $G$-Sasaki metric with constant transverse scalar curvature. Then for any convex $W$-invariant piecewise linear function $f$ on $P$, we have
\begin{eqnarray*}\mathcal L(f)\geq0.\end{eqnarray*}
Moreover, the equality holds if and only if
\begin{eqnarray*}f(v)=a^iv_i\end{eqnarray*}
for some $a=(a^i)\in\mathfrak a'_z$.
\end{prop}

\begin{proof}
As before, we assume that $\gamma$ is chosen such that $P$ contains $O$. A convex $W$-invariant piecewise linear function $f$ on $P$ can be written as
$$f=\max_{1\leq N\leq N_0}\{f_N\},$$
where $f_N$ is $W$-invariant such that
$$f_N|_{P_+}(v)=a_{N}^iv_i+c_{N}$$
for some constant vector $a_N=(a_{N}^i)$. It is showed that $a_{N}\in\overline{\mathfrak a'_+}$ (cf. \cite[Proposition 3.4]{LZZ}). Then we can divide $P_+$ into $\tau_0$ sub-polytopes $P_1,...,P_{\tau_0}$ such that for each $\tau=1,...,\tau_0$, there is an $N(\tau)\in\{1,...,N_0\}$ with
$$f|_{P_\tau}=f_{N(\tau)}.$$
For simplicity, we write $f_\tau$ as $f_{N(\tau)}$.

On the other hand,  we may write  a $G$-Sasaki metric with constant transverse scalar curvature as
$\omega_g^T=\sqrt{-1}\partial\bar\partial\varphi_0$, where $\varphi_0$ is a $K\times K$-invariant function \cite{Boyer-Galicki-Simanca}.   By \eqref{+S}, we have
\begin{eqnarray}\label{0601}
\begin{aligned}
S^T(u_0)&=-{\frac1\pi}\left((u_0^{ij}\pi)_{,ij}+{\frac\partial{\partial v_i}\left(\pi\left.{\frac{\partial \chi}{\partial x^i}}\right|_{x=\nabla u_0}\right)}\right)\notag\\
&=\bar S.
\end{aligned}
\end{eqnarray}
Then, on each $P_\tau$,
\begin{eqnarray}\label{0602}
\begin{aligned}
&-\bar S\int_{P_\tau}f\pi\, dv\\&=\int_{P_\tau} \left((u_0^{ij}\pi)_{,ij}+{\frac\partial{\partial v_i}}\left(\pi\left.{\frac{\partial \chi}{\partial x^i}}\right|_{x=\nabla u_0}\right)\right)f\,dv.
\end{aligned}
\end{eqnarray}
Note that $f_{,ij}=0$ on each $P_\tau$. Taking integration by parts, we get
\begin{eqnarray*}
\begin{aligned}
\int_{P_\tau}(u_0^{ij}\pi)_{,ij}f\pi dv&=\int_{\partial P_\tau}(u^{ij}_{0,j}\nu_i\pi+u_0^{ij}\pi_{,i}\nu_j) f\,d\sigma_0\\
&-\int_{\partial P_\tau}u_0^{ij}\nu_if_{,j}\pi\,d\sigma_0
\end{aligned}
\end{eqnarray*}
and
\begin{eqnarray*}
\begin{aligned}
\int_{P_\tau}{\frac\partial{\partial v_i}}\left(\pi\left.{\frac{\partial \chi}{\partial x^i}}\right|_{x=\nabla u_0}\right)f\,d\sigma_0=
&\int_{\partial P_\tau}\nu_i\left.{\frac{\partial \chi}{\partial x^i}}\right|_{x=\nabla u_0}f\pi\,d\sigma_0\\&-\int_{p_\tau}\left.{\frac{\partial \chi}{\partial x^i}}\right|_{x=\nabla u_0}f_{,i}\pi\,dv.
\end{aligned}
\end{eqnarray*}
Plugging the above relations into \eqref{0602}, it follows
\begin{eqnarray*}
\begin{aligned}
-\bar S\int_{P_\tau}f\pi\, dv
&=\int_{\partial P_\tau} \left(u_{0,j}^{ij}\nu_i\pi+u_0^{ij}\pi_{,i}\nu_j+\nu_i\pi\left.{\frac{\partial \chi}{\partial x^i}}\right|_{x=\nabla u_0}\right)f\,d\sigma_0\\
&-\int_{\partial P_\tau}u_0^{ij}\nu_if_{,j}\pi\,d\sigma_0-\int_{P_\tau}\left.{\frac{\partial \chi}{\partial x^i}}\right|_{x=\nabla u_0}f_{,i}\pi\,dv.
\end{aligned}
\end{eqnarray*}
Thus summing over $\tau$, using \eqref{boundary behavior of u^{ij}} and the argument of \cite[Proposition 2.2]{ZZ1},
we obtain
\begin{eqnarray}\label{0603}
\begin{aligned}
-\bar S\int_{P_+}f\pi\, dv&=\sum_{\tau_1<\tau_2}\int_{\partial P_{\tau_1}\cap\partial P_{\tau_2}}{\frac{u_0^{ij}(a_{\tau_1}^i-a_{\tau_2}^i)(a_{\tau_1}^j-a_{\tau_2}^j)}{|a_{\tau_1}-a_{\tau_2}|}}\pi\,d\sigma_0\\
&-\sum_A\Lambda_A\int_{\mathfrak F'_A\cap\partial P_+}f\langle v,\nu_A\rangle\pi\,d\sigma_0-\sum_\tau\int_{P_\tau}\left.{\frac{\partial \chi}{\partial x^i}}\right|_{x=\nabla u_0}a_{\tau}^i\pi\,dv.
\end{aligned}
\end{eqnarray}

Recall \eqref{L(u)}.  We see that
\begin{eqnarray}\label{0604}
\begin{aligned}
V_P\cdot\mathcal L(f)&=\sum_A\Lambda_A\int_{\mathfrak F'_A\cap\partial P_+}f\langle v,\nu_A\rangle\pi\,d\sigma_0-\bar S\int_{P_+}f\pi\,dv\\
&-4\sum_\tau\int_{P_\tau}\sigma(a_{\tau})\pi\,dv.
\end{aligned}
\end{eqnarray}
Note that for any $a_\tau=(a_{\tau}^i)\in\overline{\mathfrak a'_+}$,
\begin{eqnarray*}
-a_\tau^i{\frac{\partial \chi}{\partial x^i}}-4\sigma_ia_\tau^i=2\sum_{\alpha\in\Phi_+}(\coth\alpha(x)-1)\alpha(a_\tau)\geq0,~\forall~x\in\mathfrak a_+.
\end{eqnarray*}
Hence, plugging \eqref{0603} into \eqref{0604}, we derive
\begin{align}\label{positive-linear}
V_P\cdot\mathcal L(f)&=\sum_{\tau_1<\tau_2}\int_{\partial P_{\tau_1}\cap\partial P_{\tau_2}}{\frac{u_0^{ij}(a_{\tau_1}^i-a_{\tau_2}^i)(a_{\tau_1}^j-a_{\tau_2}^j)}{|a_{\tau_1}-a_{\tau_2}|}}\pi\,d\sigma_0\notag\\
&+2\sum_\tau\sum_{\alpha\in\Phi_+}\int_{P_\tau}(\coth\alpha(x)-1)\alpha(a_\tau)\pi\,dv\geq0.
\end{align}

It is easy to see that the equality in (\ref{positive-linear}) holds if and only there is an $a=(a^i)\in\overline{\mathfrak a'_+}$ such that
\begin{eqnarray}\begin{aligned}
a_\tau=a,~\forall~\tau\notag
\end{aligned}\end{eqnarray}
and
\begin{eqnarray}\begin{aligned}
\alpha(a)=0,~\forall~\alpha\in\Phi_+.\notag
\end{aligned}\end{eqnarray}
The second relation means that $a\in\mathfrak a_z'$. The proposition is proved.
\end{proof}

\begin{proof}[Proof of necessary part of Theorem \ref{main thm}]
Suppose that \eqref{0316} does not hold. Choosing $\gamma=\gamma_0$. Then
\begin{eqnarray*}
bar(P_+)-{\frac2{n+1}}\sigma\not\in\Xi.
\end{eqnarray*}
We will follow a way in \cite[Lemma 3.4]{LZZ} to construct a piecewise linear function. By \eqref{5.8+}, we may assume
$$bar(P_+)-{\frac2{n+1}}\sigma\in\mathfrak (a')_{ss}^*,$$
otherwise the Futaki invariant does not vanishes. Let $\{\alpha_{(1)},...,\alpha_{(r')}\}$ be the simple roots in $\Phi_+$. Without loss of generality, we can write
$$bar(P_+)-{\frac2{n+1}}\sigma=\lambda_1\alpha_{(1)}+...+\lambda_{r'}\alpha_{(r')},$$
where $\lambda_1\leq0$. Let $\{\varpi_i\}$ be the fundamental weights for $\{\alpha_{(1)},...,\alpha_{(r')}\}$ such that ${\frac{2\langle \varpi_i,\alpha_{(j)}\rangle}{|\alpha_{(j)}|^2}}=\delta_{ij}$. Define a $W$-invariant rational piecewise linear function $f$ on $P$ by
$$f(v)=\max_{w\in W}\{\langle w\cdot \varpi_1,v\rangle\}.$$
Then
$$f|_{P_+}=\langle \varpi_1,v\rangle.$$
Note that $\varpi_1\in\mathfrak   (a')_{ss}^*$. However,
\begin{eqnarray*}
\mathcal L(f)=n(n+1)|\alpha_{(1)}|^2\lambda_1\leq0.
\end{eqnarray*}
This contradicts to Proposition \ref{ness}. Hence \eqref{0316} is true.
\end{proof}

To prove the sufficient part of Theorem \ref{main thm}, we need the following lemma.

\begin{lem}\label{J normalization}
For any $\psi\in\mathcal H_{K\times K}\left(\frac{1}{2}d\eta\right)$ with $u_\psi\in\hat{\mathcal C}_{P,W}$, there exists a uniform constant $C$ such that
\begin{eqnarray*}
\left|J(\psi)-{\frac{1}{V_P}}\int_{P_+} u_\psi dv \right|\leq C.
\end{eqnarray*}
\end{lem}

\begin{proof}First by Lemma \ref{KAK int}, we have
\begin{eqnarray*}
J(\psi)={\frac{1}{V}}\int_M\psi\,(d\eta)^n\wedge\eta+{\frac{1}{V_P}}\int_{P_+}(u_\psi-u_0)\pi\,dv.
\end{eqnarray*}
Then the lemma is reduced to prove
\begin{eqnarray}\begin{aligned}\label{normalized bounded}
\left|\int_M\psi\,(d\eta)^n\wedge\eta\right|\le C,~\forall~u_\psi\in\hat{\mathcal C}_W.\end{aligned}\end{eqnarray}

By the normalized condition, it follows
\begin{eqnarray*}
\nabla u_{\psi}(O)=O,~u_{\psi}(O)=0.
\end{eqnarray*}
Thus
\begin{eqnarray*}
\psi(O)=-\varphi_0(O).
\end{eqnarray*}
On the other hand, since $\psi$ is a basic function,
\begin{eqnarray*}
\bigtriangleup_B\psi=\bigtriangleup_g\psi.
\end{eqnarray*}
This means that the basic Laplace operator coincides with the Laplace operator of $g$  on $\psi$. Thus by using the above two estimates and following the Green function argument in \cite[Lemma 2.2]{ZZ}, we can obtain a uniform $C_0$ such that
\begin{align}\label{average-lower}
{\frac1{V_P}}\int_M\psi(d\eta)^n\wedge\eta\geq\sup_M\psi-C_0\geq -\psi_0(O)-C_0.
\end{align}

On the other hand, by $\xi(\psi)=0$, we have
\begin{eqnarray*}
|\nabla \mathbf \psi|=|\nabla\mathbf \psi|_{Orb_M(p)}|.
\end{eqnarray*}
It follows that
\begin{eqnarray*}
|\nabla\mathbf \psi|\leq|\nabla\mathbf \varphi_0|+|\nabla \varphi|\leq2\text{diam}(P).
\end{eqnarray*}
Then by an argument in \cite[Lemma 2.2]{ZZ} and (\ref{average-lower}), we  get
\begin{align}\label{maximal}
\sup_M\psi\leq C'
\end{align}
for some large constant $C'$. Hence combining (\ref{average-lower}) and (\ref{maximal}), we obtain (\ref{normalized bounded}).
\end{proof}

\begin{proof}[Proof of sufficient part of Theorem \ref{main thm}]
First, we note that (\ref{0316+}) is equivalent to (\ref{0316}) by the relation (\ref{0306}). On the other hand,
by Lemma \ref{normalization of u}, we see that
there is a $\sigma\in Z(K')$ such that $ \hat u\in \hat{\mathcal C}_{P,W}$ for any $\psi\in\mathcal H_{K\times K}\left(\frac{1}{2}d\eta\right)$, where $\hat u$ is the Legendre function of $\varphi_{\psi_{\sigma}}$. Then by Proposition \ref{main-proposition} and Lemma \ref{J normalization}, we get
\begin{align}\label{proper-maximal-group} \mathcal K(\psi)=\mu( \hat u) &\ge \delta \int_{P_+} \hat u \pi(y)\,dy -C_\delta\notag\\
&\ge \delta J(\psi_\sigma)-C_\delta'\notag\\
&\ge \delta \inf_{\tau\in Z(H)} J(\psi_\tau)-C_\delta'.
\end{align}
(\ref{proper-maximal-group}) means that $\mathcal K(\cdot)$ is proper on $\mathcal H_{K'\times K'}\left(\frac{1}{2}d\eta\right)$ modulo $Z(H)$. Hence, by Proposition \ref{proper implies existence}, we prove the existence of $G$-Sasaki Einstein metrics.
\end{proof}

\begin{proof}[Proof of Corollary \ref{corollary-strong-sasaki}] By the necessary part of Theorem \ref{main thm},
(\ref{0316+}) holds. Then as in the proof for the sufficient part of Theorem \ref{main thm} above, for any $\psi\in\mathcal H_{K\times K}\left(\frac{1}{2}d\eta\right)$, (\ref{conjecture-strong}) holds with $Z'(T^{c})$ chosen as $Z(H)$. The corollary is proved.

\end{proof}

\section{  $G$-Sasaki Ricci solitons }

In this section, we give a version of Theorem \ref{main thm} for the existence problem of transverse Sasaki-Ricci solitons. As a generalization of transverse Sasaki-Einstein metrics, a Sasaki metric $(M, \frac{1}{2}d\eta)$ is called a \emph{transverse Sasaki-Ricci soliton} if there is an $X\in\mathfrak{ham}(M)$ such that (cf. \cite{Futaki-Ono-Wang, Martelli-Sparks-Yau, Martelli-Sparks-Yau 2, Boyer}, etc.)
$${\rm Ric}^T(g)-2(n+1)\omega_g^T=L_X\omega_g^T.$$
Clearly, ${\frac12}[d\eta ]_B={\frac{\pi}{n+1}}c_1^B(M)$ by the definition.
It has been showed that on a compact Sasaki manifold the soliton vector $X$ is determined by vanishing of the modified Futaki invariant (cf. \cite[Proposition 5.3]{Futaki-Ono-Wang}),
\begin{eqnarray}\label{mod fut}
{\rm Fut}_X(Y)=-\int_M u_Ye^{u_X}(\frac{1}{2}d\eta)^n\wedge\eta,~\forall v\in\mathfrak{ham}(M).
\end{eqnarray}
On a $G$-Sasaki manifold, by restricting the metric to the $H_0$-orbit as in Section 5, (\ref{mod fut}) is equivalent to
\begin{eqnarray}\label{X-potential}
\int_{P_+}Y^iv_ie^{X^iv_i}\pi\,dv=0,~\forall~ Y=(Y^i)\in\mathfrak z(\mathfrak h_0).
\end{eqnarray}
In particular, $X=(X^i)\in\mathfrak z(\mathfrak h_0)$.

Define a weighted barycentre with respect to $X$ by
$$bar_X(\mathcal P_+)=\frac{\int_{\mathcal P_+}y e^{X^iy_i}\pi\,d\sigma_c}{\int_{\mathcal P_+} e^{X^iy_i}\pi\,d\sigma_c}.$$
We get a soliton version of Theorem \ref{main thm} as follows.

\begin{theo}\label{soliton theorem}
Let $(M, \frac{1}{2}d\eta)$ be a $(2n+1)$-dimensional compact $G$-Sasaki manifold with $\omega_g^T\in
\frac{\pi}{n+1}c_1^B(M)>0$. Then $M$ admits a transverse Sasaki-Ricci soliton if and only if
$bar_X(\mathcal P_+)$ satisfies
\begin{align}\label{0316=}
bar_X(\mathcal P_+)-{\frac{2}{n+1}}\sigma+{\frac{1}{n+1}}\gamma_0\in\Xi.
\end{align}
\end{theo}

Similar to K\"ahler geometry, one can introduce a modified K-energy on $\mathcal H\left(\frac{1}{2}d\eta\right)$ as in \cite{TZ02, CTZ, WZZ, LZZ}, etc.. We note that an analogy of Theorem \ref{soliton theorem} for K\"ahler-Einstein $G$-manifolds has been recently estibalished in \cite{Del3} and \cite{ LZZ}, respectively. By following the argument in \cite{ LZZ}, one can extend the proof of Theorem \ref{main thm} to Theorem \ref{soliton theorem} by taking $f_a(v)=f_X(v)=e^{X^iv_i}$ in Theorem \ref{proper principle}. We left the details to the reader.

\subsection{Deformation of transverse Sasaki-Ricci solitons }
In \cite{Martelli-Sparks-Yau, Martelli-Sparks-Yau 2}, Martelli, Sparks and Yau introduced the deformation theory of Reeb vector fields $\xi$ on a compact Sasaki manifold. They showed that the volume of $M$  in fact depends only on $\xi$. Moreover, they proved that under the restriction of (\ref{transverse fano}) the Sasaki structure $(M,g,\xi,\eta)$ has the vanishing Futaki invariant if $\xi$ is a critical point of ${\rm Vol}(M,g)$. In particular, by applying their theory together with the Futaki-Ono-Wang's result for the existence of transverse Sasaki-Ricci solitons on toric Sasaki manifolds, one will obtain a deformation theorem for transverse toric Sasaki-Ricci solitons. We want to extend such a  theorem to $G$-Sasaki Ricci solitons. However, unlike the toric Sasaki manifolds, we need to overcome the obstruction condition
(\ref{0316=}).

Analogous to \cite{Martelli-Sparks-Yau}, we deform $\xi$ in $\mathfrak z(\mathfrak k)$ and see that $\xi$ must be in an open convex cone
$$\Sigma=\mathfrak C^\vee\cap\mathfrak a_z,$$
where $\mathfrak C^\vee$ is the interior of the dual cone of $\mathfrak C$. Fix a $\xi'\in\Sigma$, by Proposition \ref{sym potential}, there is a function $\rho_{\xi'}$ defined on $Z$ such that
\begin{itemize}
\item[(1)] $F_{\xi'}=\frac12\rho^2_{\xi'}$ is the Legendre function of $U_0^{\xi'}$;
\item[(2)] $\omega'=\sqrt{-1}\partial\bar\partial F_{\xi'}$ is a K\"ahler cone metric on $Z$. Thus $\{\rho_{\xi'}=1\}\cap Z$ is a toric Sasaki manifold.
\end{itemize}
Note that the complex structure of $C(M)$ does not change. By Proposition \ref{g-sasaki-structure} we see that there is a $G$-Sasaki manifold $M'$, which is diffeomorphic to $M$ and whose K\"ahler cone is $(C(M),\omega')$. Hence we get a Sasaki structure $(M,g',\xi',\eta')$ \cite[Section 3]{Boyer-Galicki 2006}.

By Proposition \ref{Fano condition}, we see that the Sasaki structure $(M,g',\xi',\eta')$ satisfies (\ref{transverse fano}) if and only if $\xi'\in \Sigma_O$, where $\Sigma_O$ is defined by (\ref{xi-condition}).

\begin{proof}[Proof of Theorem \ref{soliton deformation}]
By a change of variables $v=\iota^*(y)$, (\ref{X-potential}) is equivalent to
\begin{eqnarray}\label{X condition}
 (bar_X(\mathcal P_+)+\frac1{n+1}\gamma_0)(Y)=0,~\forall ~Y\in\mathfrak a_z.
\end{eqnarray}
Choose coordinates $y_1,...,y_{r+1}$ on $\mathfrak a_z$ such that $y_1,...,y_r$ are the coordinates on ${\rm ker}(\gamma_0)$. Then (\ref{X condition}) is equivalent to
\begin{eqnarray*}
\Psi_i(X,\xi):=\int_{\mathcal P_+}y_ie^{X^ky_k}\pi\,d\sigma_c=0,~i=1,...,r.
\end{eqnarray*}
Taking derivatives of the above $\Psi_i$'s with respect to $X^1,...,X^r$, we have
\begin{eqnarray*}
\frac{\partial \Psi_i}{\partial X^j}=\int_{\mathcal P_+}y_iy_je^{X^ky_k}\pi\,d\sigma_c,
\end{eqnarray*}
which is a strictly positive definite $(r\times r)$-matrix \cite[Lemma 2.2]{TZ02}. Since $\Xi$ is open in $\mathfrak a_{+,ss}$, the condition (\ref{0316=}) will keep on when $\xi\in\Sigma_O$ is sufficiently close to $\xi_0$. Hence the theorem follows from Theorem \ref{soliton theorem} immediately.
\end{proof}

\section{Examples}\label{examples}

In this section, we give several  examples of $G$-Sasaki manifolds and verify the existence of $G$-Sasaki Einstein metrics
or $G$-Sasaki Ricci solitons on them.

\begin{exa}\label{general exa}
Let $(M',\omega')$ be a  Fano manifold and $M$ the Kobayashi regular principle $S^1$-bundle over $M'$. Then $M$ is a regular Sasaki manifold.
\end{exa}

The Kobayashi regular principle $S^1$-bundle over a K\"ahler manifold was constructed in \cite{Kobayashi 1}. Boyer-Galicki \cite[Theorem 7.5.2]{Boyer} showed that $M$ is a regular Sasaki manifold with whose Reeb field is induced by the corresponding $S^1$-action. Furthermore, the contact form $\eta$ satisfies $\frac{1}{2} d\eta=\pi^*\omega'$ (cf. \cite[Sect. 6.7.2]{Blair}, \cite{Hatakeyama}), where $\pi$ is the projection to $M'$. Thus $M$ admits a Sasaki-Einstein metric if $(M',\omega')$ admits a K\"ahler-Einstein metric (cf. \cite[Corollary 2.1]{Boyer-Galicki 2000}).

Moreover, if $M'$ is a Fano compactification of a connect reductive group $ H$ and the $H\times H$-action can be lifted to $M$ as a bundle isomorphism,    $H\times H$ is  a  subgroup of ${\rm Aut}^T(M)$. Thus $M$ is a $G$-Sasaki manifold with
\begin{eqnarray*}
G=(H \times \mathbb C^*)\slash   {\rm diag}( H\cap \mathbb C^*).
\end{eqnarray*}

\begin{exa}\label{product exa}
Let $(M_i^{2n_i+1},g_i,\xi_i),i=1,2$ be two compact Sasaki manifolds and $(C(M_i^{2n_i+1}),\bar g_i)$ be their K\"ahler cones, respectively. Let $\omega_{\bar g_i}=\frac{\sqrt{-1}}2\partial\bar\partial \rho_i^2$  be their corresponding K\"ahler cone metrics. Take $\rho=\sqrt{\rho^2_1+\rho^2_2}$ on the product $C(M_1)\times C(M_2)$ and let $M=\{\rho=1\}$ be the corresponding level set. Then $\bar g$ is a K\"ahler metric associated to $\omega=\frac{\sqrt{-1}}2\partial\bar\partial \rho^2$ and $(M,g=\bar g|_M)$ is a Sasaki manifold.
\end{exa}

It can be verified that $\xi=\xi_1+\xi_2$ is the Reeb field of $(M,g)$. If we further assume that each $M_i$ is a $G_i$-Sasaki manifold, then it is obvious that $M$ is a $G_1\times G_2$-Sasaki manifold. Furthermore, the moment cone of $(M,g)$ is given by
\begin{eqnarray*}
\mathfrak C=\mathfrak C_1\times \mathfrak C_2,
\end{eqnarray*}
where $\mathfrak C_i$ is the moment cone of $(M_i,g_i)$. The normal vectors of facets of $\mathfrak C$ are all given by $u_{A(i)}$, where $u_{A(i)}$'s are normals of facets of $\mathfrak C_i$, considered as vectors in the product space. Thus if $\omega_{g_i}^T\in  \frac{\pi}{n_i+1} c_1^B(M_i)$, then $\omega_{g}^T\in  \frac{\pi}{n_1+n_2+2} c_1^B(M)$. Moreover, $\gamma_0=\gamma_{01}+\gamma_{02}$, where $\gamma_0\in(\mathfrak a_{1z}^*+\mathfrak a_{2z}^*)$, $\gamma_{0i}\in\mathfrak a_{iz}^*$ are determined in Proposition \ref{Fano condition} with respect to $M, M_i$, respectively.

The characteristic polytope of $(M,g)$ is given by
\begin{eqnarray*}
\mathcal P&=&\{y=(y_1,y_2)|~\xi(y)=1\}\\
&=&\cup_{t\in[0,1]}\{t\mathcal P_1+(1-t)P_2\},
\end{eqnarray*}
where $\mathcal P_i$ is the characteristic polytope of $(M_i,g_i)$ embedded in the product cone $\mathfrak C$. Then we have
\begin{eqnarray*}
bar(\mathcal P)&=&\frac{\int_{[0,1]\times \mathcal P_1\times\mathcal P_2}(ty_1,(1-t)y_2)t^{n_1}(1-t)^{n_2}\pi_1(y_1)\pi_2(y_2)}{\int_{[0,1]\times \mathcal P_1\times\mathcal P_2}t^{n_1}(1-t)^{n_2}\pi_1(y_1)\pi_2(y_2)}\\
&=&\left(\frac{n_1+1}{n_1+n_2+2}bar(\mathcal P_1),\frac{n_2+1}{n_1+n_2+2}bar(\mathcal P_2)\right).
\end{eqnarray*}
Thus $M$ admits a transverse $G$-Sasaki Einstein metric if and only if both $M_i$ do.  Hence, by Theorem \ref{soliton deformation}, we may deform to a family of non-product transverse $G$-Sasaki Ricci solitons from a product transverse $G$-Sasaki Einstein metric $(M,\xi)$.

\begin{exa}\label{exa1}
Let $K=U(2)$ and $G=GL_2(\mathbb C)$. Identify $\mathbb C^4\backslash\{O\}$ with the set of non-zero $2\times2$ complex matrixes $M_{2\times2}(\mathbb C)\backslash\{O\}$. For any $A\in \mathbb C^4\backslash\{O\}$, define $$\rho^2(A)={\rm tr}(A\bar A^T).$$
Then we get a $GL_2(\mathbb C)$-Sasaki manifold
$$S^7(1)=\{A\in \mathbb C^4\backslash\{O\}|\rho(A)=1\},$$
which is the standard Euclidean sphere.
\end{exa}

\begin{figure}[hbp]
\centering
\includegraphics[height=1.5in]{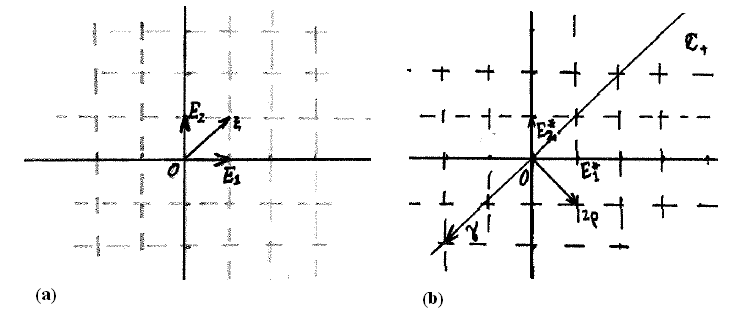}
\caption{The lattice (a) $\mathfrak a$ and (b) $\mathfrak a^*$.}
\end{figure}

It is easy to see that ${\frac{\sqrt{-1}}2}\partial\bar\partial\rho^2$ is the standard Euclidean metric on $\mathbb C^4$, thus $S^7(1)$ is the standard unit sphere. In the following, we verify that $S^7(1)$ is a $G$-Sasaki Einstein metric. We consider the $GL_2(\mathbb C)\times GL_2(\mathbb C)$ action on $\mathbb C^4\backslash\{O\}$ given by
\begin{eqnarray*}
(GL_2(\mathbb C)\times GL_2(\mathbb C))\times (\mathbb C^4\backslash\{O\})&\to&\mathbb C^4\backslash\{O\}\\
((X_1,X_2),Z)&\to&X_1ZX_2^{-1}.
\end{eqnarray*}
Then $\mathbb C^4\backslash\{O\}$ satisfies Definition \ref{definition} (1). Obviously, $\rho$ is $K\times K$-invariant. By a direct computation, we have
$$\mathfrak g=\text{Span}_{\mathbb C}\left\{\left(\begin{aligned}1& &0\\0& &0\end{aligned}\right),\left(\begin{aligned}0& &0\\1& &0\end{aligned}\right),\left(\begin{aligned}0& &1\\0& &0\end{aligned}\right),\left(\begin{aligned}0& &0\\0& &1\end{aligned}\right)\right\}$$
and
$$\mathfrak t^c=\text{Span}_{\mathbb C}\left\{\left(\begin{aligned}1& &0\\0& &0\end{aligned}\right),\left(\begin{aligned}0& &0\\0& &1\end{aligned}\right)\right\}.$$
The Reeb  vector field $\xi$ is given by $\xi=\left(\begin{aligned}\sqrt{-1}& &0\\0& &\sqrt{-1}\end{aligned}\right)$, which satisfies Definition \ref{definition} (3).

We choose a maximal torus
\begin{eqnarray*}
T^c=\left\{\left(\begin{aligned}e^z& &0\\0& &e^{-z} \end{aligned}\right)|z\in \mathbb C^*\right\}.
\end{eqnarray*}
Then the restriction of ${\frac12}\rho^2$ on it is given by
\begin{eqnarray*}
{\frac12}\rho^2(z)={\frac12}(|e^z|^2+|e^z|^{-2}).
\end{eqnarray*}
Choose $E_1,E_2$ as the generators of $\mathfrak a$ and $E_1^*$, $E_2^*$ be their dual in $\mathfrak a_*$, we see that the lattice of characters of $G$ is generated by
$E_1^*$ and $E_2^*$ (See Fig-1).

A direct computation shows that
$$\mathfrak C=\{y_1E_1^*+y_2E_2^*\in \mathfrak a^*|~y_1,y_2\geq0\}.$$
Also, we have
\begin{eqnarray*}
2\sigma=E_1^*-E_2^*,& &\xi=E_1+E_2,\\
u_1=E_1,& &u_2=E_2,\\
\gamma_0=-2(E_1^*+E_2^*).& &
\end{eqnarray*}
Then one can check that (\ref{0316}) holds. In fact, $S^7(1)$ can be regarded as a Hopf $S^1$-fiberation, which is  a $S^1$-bundle over $\mathbb{CP}^3$, and $\mathbb{CP}^3$ is a Fano compactification of $PGL_2(\mathbb C)$ (cf. \cite[Example 2.2]{AK}). However, the $PGL_2(\mathbb C)\times PGL_2(\mathbb C)$-action can not be lifted to $\mathbb C^4$, so it is not of the kind given in Example \ref{general exa}.

\begin{exa}\label{exa2}
Let $K=SU(2)\times S^1$ and $G=SL_2(\mathbb C)\times \mathbb C^*$. Identify $\mathbb C^5\backslash\{O\}$ with $(M_{2\times2}(\mathbb C)\oplus\mathbb C)\backslash\{O\}$. Consider the hypersurface $$\mathcal H:=\{(A,t)\in\mathbb C^5\backslash\{O\}|\det(A)=t^2\}.$$
For any $(A,t)\in \mathbb C^5\backslash\{O\}$, define $$\rho^2(A,t):={\frac16}\left(\text{tr}(A\bar A^T)+|t|^2\right).$$ Then $M=\mathcal H\cap \{\rho=1\}$ is an  $(SL_2(\mathbb C)\times\mathbb C^*)$-Sasaki manifold of dimension $7$, whose K\"ahler cone is $\mathcal H.$
\end{exa}

As in Example \ref{exa1}, ${\frac{\sqrt{-1}}2}\partial\bar\partial\rho^2$ is the standard Euclidean metric on $\mathbb C^5$ and $M$ is the intersection of $\mathcal H$ and the unit sphere. Consider the $G\times G$-action on $\mathcal H$ given by
\begin{eqnarray*}
(G\times G)\times \mathcal H&\to&\mathcal H\\
((X_1,t_1),(X_2,t_2),(A,t))&\to&(t_1X_1At_2^{-1}X_2^{-1},t_1tt_2^{-1}).
\end{eqnarray*}
one can check directly that  (1)-(2) in Definition  \ref{definition} are satisfied.

By a direct computation, we have
$$\mathfrak g=\text{Span}_{\mathbb C}\left\{\left(\left(\begin{aligned}1& &0\\0& &-1\end{aligned}\right),0\right),\left(\left(\begin{aligned}0& &0\\1& &0\end{aligned}\right),0\right),\left(\left(\begin{aligned}0& &1\\0& &0\end{aligned}\right),0\right),\left(\left(\begin{aligned}0& &0\\0& &0\end{aligned}\right),1\right)\right\},$$
and
$$\mathfrak t^c=\text{Span}_{\mathbb C}\left\{\left(\left(\begin{aligned}1& &0\\0& &-1\end{aligned}\right),0\right),\left(\left(\begin{aligned}0& &0\\0& &0\end{aligned}\right),1\right)\right\}.$$
Let $\xi=\left(\left(\begin{aligned}0& &0\\0& &0\end{aligned}\right),\sqrt{-1}\right)$. Then $\xi\in\mathfrak z(\mathfrak k)$ and $M$ is a $G$-Sasaki manifold with the Reeb vector field $\xi$.

Let us determine the moment cone of this Sasaki manifold. Choose a maximal torus
\begin{eqnarray*}
T^c=\left\{\left(\left(\begin{aligned}e^z& &0\\0& &e^{-z} \end{aligned}\right)t,t\right)|z,t\in \mathbb C^*\right\}.
\end{eqnarray*}
Then the restriction of ${\frac12}\rho^2$ on it is
\begin{eqnarray*}
{\frac12}\rho^2(z,t)={\frac16}(|e^z|^2+|e^z|^{-2}+1)|t|^2.
\end{eqnarray*}
Choose $E_1,E_2$ as the generators of $\mathfrak a$ and $E_1^*$, $E_2^*$ be their dual in $\mathfrak a_*$. We see that the lattice of characters of $G$ is generated by
$e_1^*={\frac12}(E_1^*+E_2^*)$ and $e_2^*={\frac12}(E_1^*-E_2^*)$ (See Fig-2).

\begin{figure}[hbp]
\centering
\includegraphics[height=1.5in]{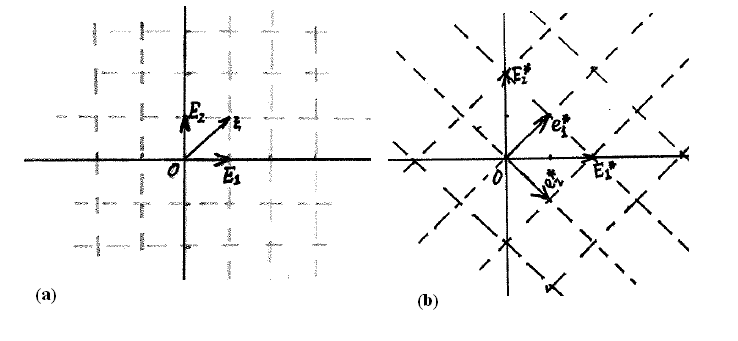}
\caption{The lattice (a) $\mathfrak a$ and (b) $\mathfrak a^*$.}
\end{figure}

A direct computation shows that
$$\mathfrak C=\{y_1e_1^*+y_2e_2^*\in \mathfrak a^*|-y_1+y_2\geq0,y_1+y_2\geq0\},$$
and the positive part is of
$$\mathfrak C_+=\{y_1e_1^*+y_2e_2^*\in \mathfrak a^*|-y_1+y_2\geq0,y_1\geq0\}.$$
Also, we have
\begin{eqnarray*}
2\sigma=2e_2^*,\\
u_1=e_1+e_2,& &u_2=e_1-e_2,\\
\gamma_0=-3e_1^*.& &
\end{eqnarray*}

On the other hand, $\xi=E_1+E_2$, and so $\gamma_0(\xi)=-3$. Thus $\xi$ does not define a Sasaki structure such that the corresponding transverse K\"ahler form lies in ${\frac\pi{n+1}}c_1^B(M)$. But by replacing $\xi$ by $\xi'={\frac43}\xi$, we get a Sasaki structure on $M$ whose transverse K\"ahler form lies in ${\frac\pi{n+1}}c_1^B(M)$. In fact, this new Sasaki structure can be derived from the original one by applying a $\mathcal D$-homothetic deformation defined by Tanno \cite{Tanno} (see also \cite{Boyer-Galicki 2006}). It can be checked that (\ref{0316}) holds in this case. Thus the Sasaki manifold $M$ with its Reeb vector field $\xi'$, admits a Sasaki-Einstein metric. In fact, in this case, $M$ is an $S^1$-bundle over $M\slash e^{t\xi}$, which is the wonderful compactification of $SL_2(\mathbb C)$. It is known that the wonderful compactification of $SL_2(\mathbb C)$ admits a K\"ahler-Einstein metric .

\begin{exa}
Let $n=4$, $G=PSL_2(\mathbb C)\times\mathbb C^*$ and $\hat G=G\times\mathbb C^*$. Choose $2\sigma=(1,0,0)$ to be
a positive root in $\hat{\mathfrak g}\cong \mathbb R^3$. Let $\mathfrak C$ be the cone in $\hat{\mathfrak g}$ given by
\begin{eqnarray*}
\mathfrak C=\{y_3-y_2\geq0,~y_3+y_2\geq0,~2y_3-y_2-y_1\geq0,~2y_3-y_2+y_1\geq0~\}.
\end{eqnarray*}
Then there is a $\hat G$-Sasaki manifold of dimension 9 such that $\mathfrak C$ is its moment cone.
\end{exa}

Clearly, $\mathfrak C$ is a good cone. Moreover, its facets intersect with Weyl wall orthogonally. Thus by \cite[Proposition 2.5]{AK}, there is a smooth K\"ahler manifold $\hat M$ with an open dense $\hat G\times \hat G$-orbit isomorphic to $\hat G$. Furthermore, if we equip the toric orbit $Z$ in $\hat M$ with a toric cone metric, then it extends to a K\"ahler cone metric on $\hat M$ by Proposition \ref{g-sasaki-structure}.

\begin{figure}[hbp]
\centering
\includegraphics[height=1.5in]{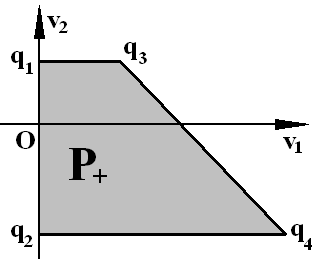}
\caption{Polytope $P_+$.}
\end{figure}

Next, we choose a possible $\xi$ such that $\hat M$ is the K\"ahler cone of a Sasaki manifold $M$ with $c_1^B(M)>0$. By Proposition \ref{Fano condition}, we see
\begin{align*}
\gamma_0&=(0,0,-1),\\
\xi&=(0,\xi^2,5),~-5<\xi^2<5.
\end{align*}
Then the polytope $P_+$ in $(\mathfrak a'_+)^*$ is the convex hull of the following four points (See Fig-3)
\begin{eqnarray}\begin{aligned}&q_1=\left(0,\frac1{\xi^2+5}\right),q_2=\left(0,\frac{-1}{-\xi^2+5}\right),\notag\\
&q_3=\left(\frac1{\xi^2+5},\frac1{\xi^2+5}\right),q_4=\left(\frac3{-\xi^2+5},\frac{-1}{-\xi^2+5}\right).\end{aligned}\end{eqnarray}

The soliton vector field on $\mathfrak a'$ is of form,
$$X=(0,\lambda)$$
for some $\lambda\in\mathbb R$. Also $\frac{2\sigma}{n+1}=\left(\frac15,0\right)$.
Let $\xi^2\to5$. Then $\lambda\to+\infty$. In this case the barycentre of $P_+$,
$$bar_X(P_+)\to\left({\frac3{40}},0\right).$$
Thus $M$ admits no Sasaki-Ricci soliton when $\xi^2$ is chosen sufficiently close to $5$.
But if $\xi^2\to-5$, we have $\lambda\to-\infty$. In this case
$$bar_X(P_+)\to\left({\frac9{40}},0\right)$$
and $M$ admits Sasaki-Ricci soliton metric when $\xi^2$ is chosen sufficiently close to $-5$.
In particular, $bar_X(P_+)=\left({\frac3{10}},0\right)$ when $\xi^2=-\frac52$. Hence we prove the existence of Sasaki-Ricci solitons on $\hat M$.


\end{document}